\documentclass[12pt, reqno]{amsart}
\usepackage{amsfonts,mathrsfs,amsmath,amssymb}

\def\1{\hbox{1\kern-.35em\hbox{1}}}

\newtheorem{theorem}{Theorem}[section]
\newtheorem*{theorem*}{Theorem}
\newtheorem{lem}[theorem]{Lemma}
\newtheorem{proposition}[theorem]{Proposition}
\newtheorem*{proposition*}{Proposition}
\newtheorem{corollary}[theorem]{Corollary}
\newtheorem{definition}[theorem]{Definition}

\newtheorem{remark}[theorem]{Remark}

\numberwithin{equation}{section}

\newcommand{\bea}{\begin{eqnarray}}
\newcommand{\eea}{\end{eqnarray}}
\newcommand{\be}{\begin{eqnarray*}}
\newcommand{\ee}{\end{eqnarray*}}

\newcommand{\Z}{{\mathbb Z}}

\newcommand{\C}{{\mathbb C}}

\def\sv{\mathfrak{sv}}
\def\tsv{\widetilde{\mathfrak{sv}}}
\def\hsv{\widehat{\mathfrak{sv}}}

\newcommand{\I}{{\rm I}}

\def\qed{\hfill\mbox{$\Box$}}
\def\epsi{\epsilon}
\def\a{\alpha}

\def\b{\beta}

\def\si{\sigma}

\def\dis{\displaystyle}

\def\Z{\mathbb{Z}}

\def\C{\mathbb{C}}

\def\vp{\varphi}

\def\tau{\pi}

\numberwithin{equation}{section}
\begin{document}
%
%%%%%%%%%%%%%%%%%%%%%%%%%%%%%%%%%%%%%%%%%%%%%%%%%%%%%%%%
\title[Structure of  the extended Schr\"{o}dinger-Virasoro Lie algebra]
{Structure of  the extended Schr\"{o}dinger-Virasoro Lie algebra
${\tsv}$$^{*}$}
\author[Gao]{Shoulan Gao}
\address{Department of Mathematics, Shanghai Jiaotong University, Shanghai
200240, China} \email{gaoshoulan@sjtu.edu.cn}

\author[Jiang]{Cuipo Jiang$^{\dag}$}
\address{Department of Mathematics, Shanghai Jiaotong University, Shanghai
200240, China} \email{cpjiang@sjtu.edu.cn}

\author[Pei]{Yufeng Pei}
\address{Department of Mathematics, Shanghai Jiaotong University, Shanghai
200240, China} \email{yfpei@sjtu.edu.cn}

\begin{abstract}
In this paper,   we study the derivations, the central extensions
and the  automorphism group of the extended
Schr\"{o}dinger-Virasoro Lie algebra  $\widetilde{{\sv}}$,
introduced by J. Unterberger \cite{U} in the context of
two-dimensional conformal field theory and statistical physics.
Moreover, we show that $\widetilde{{\sv}}$ is an
infinite-dimensional complete Lie algebra and the universal
central extension of $\widetilde{{\sv}}$ in the category of
Leibniz algebras is the same as that in the category of Lie
algebras.
\end{abstract}

\thanks{$^{*}$ Supported in part by China NSF grant
10571119.  }

\thanks{$^{\dag}$ Corresponding author: cpjiang@sjtu.edu.cn }
\maketitle

%\tableofcontents

%%%%%%%%%%%%%%%%%%%%%%%%%%%%%%%%%%%%%%%%%%%%%%%%%%%%%%%%%%%%%%%%
%
\section{\bf Introduction }
%
%%%%%%%%%%%%%%%%%%%%%%%%%%%%%%%%%%%%%%%%%%%%%%%%%%%%%%%%%%%%%%%%
\label{sub0-c-related}
The Schr\"{o}dinger-Virasoro Lie algebra ${\sv}$, originally introduced by M. Henkel
in \cite{H} during his study on the invariance of the free Schr\"{o}dinger
equation, is  a vector space over the complex field $\C$
with a basis $\{L_{n}, M_{n}, Y_{n+\frac{1}{2}}\ |\ n\in\Z\}$ and
the Lie brackets:
\begin{eqnarray*}
&&[L_{m}, L_{n}]=(n-m)L_{m+n}, \ \ [L_{m}, M_{n}]=nM_{m+n},  \ \ [L_{m},
Y_{n+\frac{1}{2}}]=(n+\frac{1-m}{2})Y_{m+n+\frac{1}{2}},\\
&& [M_{m},
M_{n}]=0, \ \ [Y_{m+\frac{1}{2}},
Y_{n+\frac{1}{2}}]=(m-n)M_{m+n+1},\ \ [M_{m},
Y_{n+\frac{1}{2}}]=0,
\end{eqnarray*}
 for all $m,n\in\Z$. It is easy to see that $\sv$ is a
semi-direct product of the centerless Virasoro algebra (Witt
algebra)  $\mathfrak{Vir}_0=\bigoplus_{n\in\Z}\C L_{n},$ which can
be regarded as the Lie algebra that consists of derivations on the
Laurent polynomial ring  \cite{K}, and the two-step nilpotent
infinite-dimensional Lie algebra
$\mathfrak{h}=\bigoplus_{m\in\Z}\C
Y_{m+\frac{1}{2}}\bigoplus_{m\in\Z}\C M_{m},$ which contains the
Schr\"{o}dinger Lie algebra $\mathfrak{s}$ spanned by
$\{L_{-1},L_0,L_1,Y_{-\frac{1}{2}},Y_{\frac{1}{2}},M_0\}$. Clearly
$\mathfrak{s}$ is isomorphic to the semi-direct product of the Lie
algebra $\mathfrak{sl}(2)$ and the three-dimensional nilpotent
Heisenberg Lie algebra $\langle
Y_{-\frac{1}{2}},Y_{\frac{1}{2}},M_0\rangle$. The structure and
representation theory of $\sv$ have been extensively studied by C.
Roger and J. Unterberger. We refer the reader to \cite{RU} for
more details. Recently, in order to investigate vertex
representations of $\sv$, J. Unterberger \cite{U} introduced a
class of new infinite-dimensional Lie algebras $\tsv$ called the
extended Schr\"{o}dinger-Virasoro algebra (see section 2), which
can be viewed as an extension of $\sv$ by a conformal current with
conformal weight $1$.

In this paper, we  give a complete description of the derivations,
the central extensions and the automorphisms for the extended
Schr\"{o}dinger-Virasoro Lie algebra $\widetilde{{\sv}}$. The
paper is organized as follows:

In section 2, we show that the center of ${\tsv}$ is zero and all
derivations of ${\tsv}$ are inner derivations, i.e., $\rm
H^1(\tsv,\tsv)=0,$ which implies that $\tsv$ is a complete Lie
algebra. Recall that a Lie algebra  is called complete if its
center is zero and all derivations are inner, which was originally
introduced by N. Jacobson in \cite{J}. Over the past decades, much
progress has been obtained on the theory of complete Lie algebras
( see for example \cite{JMZ,ZM,MZJ}). Note that since the center
of ${\sv}$ is non-zero and $\rm {dim\ H}^1(\sv,\sv)=3$, $\sv$ is
not a complete Lie algebra.

In section 3, we determine  the universal central extension of
${\tsv}$. The universal covering algebra $\hsv$ contains the
twisted Heisenberg-Virasoro Lie algebra, which plays an important
role in the representation theory of toroidal Lie algebras
\cite{B1,B2,JJ,SJ}, as a subalgebra. Furthermore, in section 4, we
show that there is no non-zero symmetric invariant bilinear form
on $\tsv$, which implies the universal central extension of
${\tsv}$ in the category of Leibniz algebras is the same as that
in the category of Lie algebras \cite{HPL}.

Finally in section 5,  we give  the  automorphism groups of
${\tsv}$ and its universal covering algebra $\hsv$, which are
isomorphic.

Throughout the paper, we  denote by $\Z$ and $\C^{*}$  the set of
integers and the set of non-zero complex numbers respectively, and
all the vector spaces are assumed over the complex field $\C$.

%%%%%%%%%%%%%%%%%%%%%%%%%%%%%%%%%%%%%%%%%%%%%%%%%%%%%%%%%%%%%%%%
%
\section{\bf The Derivation Algebra of $\tsv$}
%
%%%%%%%%%%%%%%%%%%%%%%%%%%%%%%%%%%%%%%%%%%%%%%%%%%%%%%%%%%%%%%%%
\label{sub3-c-related}

\begin{definition}\label{D1.1}
The extended Schr\"{o}dinger-Virasoro Lie algebra
${\tsv}$ is a vector space  spanned by a basis
$\{L_{n}, M_{n}, N_{n}, Y_{n+\frac{1}{2}}\ |\ n\in\Z\}$ with the
following brackets
$$[L_{m}, L_{n}]=(n-m)L_{m+n}, \ \ [M_{m}, M_{n}]=0, \ \ [N_{m}, N_{n}]=0,
\ \ [Y_{m+\frac{1}{2}}, Y_{n+\frac{1}{2}}]=(m-n)M_{m+n+1},$$
$$[L_{m}, M_{n}]=nM_{m+n}, \ \ [L_{m}, N_{n}]=nN_{m+n}, \ \ [L_{m},
Y_{n+\frac{1}{2}}]=(n+\frac{1-m}{2})Y_{m+n+\frac{1}{2}}, $$
$$ \ \
[N_{m}, M_{n}]=2M_{m+n},  \ \ [N_{m},
Y_{n+\frac{1}{2}}]=Y_{m+n+\frac{1}{2}}, \ \ [M_{m},
Y_{n+\frac{1}{2}}]=0,$$ for all $m,n\in\Z$.
\end{definition}
It is clear that ${\tsv}$ is a perfect Lie algebra, i.e.,
$[{\tsv}, {\tsv}]={\tsv}$ , which is finitely generated with a
set of generators
 $\{L_{-2}, L_{-1},L_{1}, L_{2}, N_{1}, Y_{\frac{1}{2}}\}$.

 Define a $\dis\frac{1}{2}\Z$-grading on ${\tsv}$ by
$$deg(L_{n})=n, \  deg(M_{n})=n, \  deg(N_{n})=n, \
deg(Y_{n+\frac{1}{2}})=n+\dis\frac{1}{2},$$ for all $n\in \Z$.
Then
$${\tsv}=\bigoplus\limits_{n\in\Z}{\tsv}_{\frac{n}{2}}
=(\bigoplus\limits_{n\in\Z}{\tsv}_{n})
\bigoplus(\bigoplus\limits_{n\in\Z}{\tsv}_{n+\frac{1}{2}}),$$
where ${\tsv}_{n}=span\{L_{n}, M_{n}, N_{n}\}$
and
 ${\tsv}_{n+\frac{1}{2}}=span\{Y_{n+\frac{1}{2}}\}$ for all $n\in\Z$.

% \begin{remark}\label{R1.2}
%  The $\dis\frac{1}{2}\Z$-grading on ${\tsv}$
% is actually given by the adjoint action of $L_{0}$.
%  \end{remark}

\begin{definition}{\rm (\cite{F})}
Let $G$ be a commutative group,
$\mathfrak{g}=\bigoplus\limits_{g\in G}\mathfrak{g}_{g}$ a
$G$-graded Lie algebra. A $\mathfrak{g}$-module $V$ is called
$G$-graded, if
$$V=\bigoplus\limits_{g\in G}V_{g\in G},
\;\;\; \mathfrak{g}_{g}V_{h}\subseteq V_{g+h},\;\;\; \forall\;
g,h\in G.$$
\end{definition}

\begin{definition}{\rm (\cite{F})}
Let $\mathfrak{g}$ be a Lie algebra and $V$  a
$\mathfrak{g}$-module. A linear map $D:
\mathfrak{g}\longrightarrow V$ is called a derivation, if for any
$x, y\in \mathfrak{g}$, we have
$$D[x, y]=x.D(y)-y.D(x).$$
If there exists some $v\in V$ such that $D:x\mapsto x.v$, then $D$
is called an inner derivation.
\end{definition}
Let $\mathfrak{g}$ be a Lie algebra,  $V$  a module of
$\mathfrak{g}$. Denote by $Der(\mathfrak{g}, V)$ the vector space of
all derivations, $Inn(\mathfrak{g}, V)$ the vector space of all
inner derivations. Set
$$H^{1}(\mathfrak{g}, V) =Der(\mathfrak{g}, V) /Inn(\mathfrak{g}, V).$$
Denote by $Der(\mathfrak{g})$ the derivation algebra of
$\mathfrak{g}$, $Inn(\mathfrak{g})$  the vector space of all inner
derivations of $\mathfrak{g}$. We will prove that  all the
derivations of ${\tsv}$ are inner derivations.

 By Proposition 1.1 in \cite{F}, we have the following lemma.

\begin{lem}\label{L3.2}
$$Der({\tsv})=\bigoplus\limits_{n\in\Z}Der({\tsv})_{\frac{n}{2}},$$
where
$Der({\tsv})_{\frac{n}{2}}({\tsv}_{\frac{m}{2}})\subseteq
{\tsv}_{\frac{m+n}{2}}$ for all $m,n\in\Z$.
\end{lem}

\qed

\begin{lem}\label{L3.3}
$H^{1}({\tsv}_{0},{\tsv}_{\frac{n}{2}})=0,\quad\forall\;n\in\Z\setminus\{0\}.$
\end{lem}

\begin{proof} We have to prove
$$H^{1}({\tsv}_{0},{\tsv}_{n})=0,
\ \ \forall n\in\Z\setminus\{0\},$$
$$H^{1}({\tsv}_{0},{\tsv}_{n+\frac{1}{2}})=0,
\ \ \forall n\in\Z.$$

 (1) For $m\neq 0$, let $\vp:{\tsv}_{0}\longrightarrow
{\tsv}_{m} $ be a derivation. Assume that
\begin{eqnarray*}
& &\vp(L_{0})=a_{1}L_{m}+b_{1}M_{m}+c_{1}N_{m},
\\
& &\vp(M_{0})=a_{2}L_{m}+b_{2}M_{m}+c_{2}N_{m},
\\
& &\vp(N_{0})=a_{3}L_{m}+b_{3}M_{m}+c_{3}N_{m},
\end{eqnarray*}
where $a_{i}, b_{i}, c_{i}\in\C, i=1, 2, 3 $. Since
$$\vp[L_{0},
M_{0}]=[\vp(L_{0}), M_{0}]+[L_{0}, \vp(M_{0})],$$ we have
$$a_{2}mL_{m}+(b_{2}m+2c_{1})M_{m}+c_{2}mN_{m}=0.$$
So $a_{2}=0, c_{1}=-\dis\frac{1}{2} b_{2}m, c_{2}=0$ and
$\vp(M_{0})=b_{2}M_{m}.$  Since
$$\vp[L_{0},
N_{0}]=[\vp(L_{0}),N_{0}]+[L_{0}, \vp(N_{0})],$$
we have
$$a_{3}mL_{m}+(b_{3}m-2b_{1})M_{m}+c_{3}mN_{m}=0.$$
Then $a_{3}=0, b_{1}=\dis\frac{1}{2}b_{3}m, c_{3}=0$ and
$\vp(N_{0})=b_{3}M_{m}.$ Therefore, we get
$$\vp(L_{0})=a_{1}L_{m}+\dis\frac{1}{2}b_{3}mM_{m}-\dis\frac{1}{2}
b_{2}mN_{m},\quad  \vp(M_{0})=b_{2}M_{m}, \quad
\vp(N_{0})=b_{3}M_{m}.$$ Let
$X_{m}=\dis\frac{a_{1}}{m}L_{m}+\dis\frac{1}{2}b_{3}M_{m}-\dis\frac{1}{2}
b_{2}N_{m}$, we have
$$\vp(L_{0})=[L_{0}, X_{m}],\quad  \vp(M_{0})=[M_{0}, X_{m}], \quad \vp(N_{0})=[N_{0}, X_{m}].$$
Then $\vp\in
Inn({\tsv}_{0},{\tsv}_{m})$.
Therefore,
$$H^{1}({\tsv}_{0},{\tsv}_{m})=0, \quad \forall\;
m\in\Z\setminus\{0\}.$$

(2) For all $m\in\Z$, let
$\vp:{\tsv}_{0}\longrightarrow
{\tsv}_{m+\frac{1}{2}} $ be a derivation.
Assume that
$$\vp(L_{0})=aY_{m+\frac{1}{2}}, \quad \vp(M_{0})=bY_{m+\frac{1}{2}}, \quad \vp(N_{0})=cY_{m+\frac{1}{2}},$$
for some $a, b, c\in\C $. Because $\vp[L_{0}, M_{0}]=[\vp(L_{0}),
M_{0}]+[L_{0}, \vp(M_{0})],$ we have
$$0=[L_{0}, \vp(M_{0})]= [L_{0}, bY_{m+\frac{1}{2}}]=b(m+\frac{1}{2})Y_{m+\frac{1}{2}}.  $$
So $b=0$ and $\vp(M_{0})=0$. Since $\vp[L_{0},
N_{0}]=[\vp(L_{0}),N_{0}]+[L_{0}, \vp(N_{0})],$ we have
$$0=[aY_{m+\frac{1}{2}},N_{0}]+[L_{0}, cY_{m+\frac{1}{2}}]=(c(m+\frac{1}{2})-a)Y_{m+\frac{1}{2}}.$$
Then $a=c(m+\frac{1}{2})$. Hence, we get
$$\vp(L_{0})=c(m+\frac{1}{2})Y_{m+\frac{1}{2}}, \quad \vp(M_{0})=0, \quad \vp(N_{0})=cY_{m+\frac{1}{2}}.$$
Letting $X_{m+\frac{1}{2}}=cY_{m+\frac{1}{2}}$, we obtain
$$\vp(L_{0})=[L_{0}, X_{m+\frac{1}{2}}],\quad  \vp(M_{0})=[M_{0}, X_{m+\frac{1}{2}}],
 \quad \vp(N_{0})=[N_{0}, X_{m+\frac{1}{2}}].$$
Then $\vp\in
Inn({\tsv}_{0},{\tsv}_{m+\frac{1}{2}})$.
Therefore,
$$H^{1}({\tsv}_{0},{\tsv}_{m+\frac{1}{2}})=0, \quad \forall\;
m\in\Z.$$
\end{proof}

\begin{lem}\label{L3.4}
$Hom_{{\tsv}_{0}}({\tsv}_{\frac{m}{2}}, {\tsv}_{\frac{n}{2}})=0$
for all $m,n\in\Z$, $m\neq n.$
\end{lem}

\begin{proof} Let $f\in
Hom_{{\tsv}_{0}}({\tsv}_{\frac{m}{2}}, {\tsv}_{\frac{n}{2}})$,
where $m\neq n$. Then for any $E_{0}\in {\tsv}_{0},
E_{\frac{m}{2}}\in {\tsv}_{\frac{m}{2}}$, we have
$$f([E_{0}, E_{\frac{m}{2}}])=[E_{0}, f(E_{\frac{m}{2}})].$$
Then $f([L_{0}, E_{\frac{m}{2}}])=[L_{0}, f(E_{\frac{m}{2}})]$,
i.e.,
$$\dis\frac{m}{2}f( E_{\frac{m}{2}})=[L_{0}, f(E_{\frac{m}{2}})]=\dis\frac{n}{2}f( E_{\frac{m}{2}}).$$
So we have $f( E_{\frac{m}{2}})=0$ for all $m\neq n$. Therefore,
we have $f=0$.
\end{proof}

By Lemma \ref{L3.3}-\ref{L3.4} and Proposition 1.2 in \cite{F}, we
have the following Lemma.

\begin{lem}\label{L3.5}
$Der({\tsv})=Der({\tsv})_{0}+Inn({\tsv}).$
\qed
\end{lem}

\begin{lem}\label{L3.6} For any $D\in Der({\tsv})_{0}$,
there exist some $a,b,c\in \C$ such that
$$D=ad(aL_{0}-\dis\frac{c}{2}M_{0}+(b-\frac{a}{2})N_{0}).$$
Therefore, $Der({\tsv})_{0} \subseteq
Inn({\tsv}).$
\end{lem}

\begin{proof} For any $D\in Der({\tsv})_{0}$, assume
that for all $m\in\Z$,
\begin{eqnarray*}
& &D(L_{m})=a^{m}_{11}L_{m}+a^{m}_{12}M_{m}+a^{m}_{13}N_{m},
\\
& &D(M_{m})=a^{m}_{21}L_{m}+a^{m}_{22}M_{m}+a^{m}_{23}N_{m},
\\
& &D(N_{m})=a^{m}_{31}L_{m}+a^{m}_{32}M_{m}+a^{m}_{33}N_{m},
\\
& &D(Y_{m+\frac{1}{2}})=b^{m+\frac{1}{2}}Y_{m+\frac{1}{2}},
\end{eqnarray*}
where $a_{ij}^{m},\; b^{m+\frac{1}{2}}\in\C, \;i, j=1, 2, 3$. For
any $E_{\frac{m}{2}}\in
{\tsv}_{\frac{m}{2}},\; E_{\frac{n}{2}}\in
{\tsv}_{\frac{n}{2}}$, we have
$$D[E_{\frac{m}{2}},E_{\frac{n}{2}} ]=[D(E_{\frac{m}{2}}),E_{\frac{n}{2}} ]+[E_{\frac{m}{2}}, D(E_{\frac{n}{2}})].$$
In particular, $D[L_{0}, E_{\frac{m}{2}}]=[D(L_{0}),
E_{\frac{m}{2}}]+[L_{0}, D(E_{\frac{m}{2}})].$ Then
$$\frac{m}{2}D(E_{\frac{m}{2}})=[D(L_{0}), E_{\frac{m}{2}}]+\frac{m}{2}D(E_{\frac{m}{2}}).$$
So $$[D(L_{0}), E_{\frac{m}{2}}]=0.$$ Since  $[D(L_{0}), L_{1}]=0,
[D(L_{0}), M_{0}]=0$ and $[D(L_{0}), N_{0}]=0$, we can deduce that
$$a^{0}_{11}=a^{0}_{12}=a^{0}_{13}=0.$$
So $D(L_{0})=0.$ Because $D[M_{0}, L_{m} ]=[D(M_{0}), L_{m}
]+[M_{0}, D(L_{m})]$, we have
$$a^{0}_{21}mL_{m}-2a^{m}_{13}M_{m}=0.$$
Then
$$a^{0}_{21}=0,\;a^{m}_{13}=0,\;D(M_{0})=a^{0}_{22}M_{0}+a^{0}_{23}N_{0}.$$
By the fact that  $D[M_{0},N_{m}]=[D(M_{0}),N_{m} ]+[M_{0},
D(N_{m})]$, we have
$$a^{m}_{21}L_{m}+a^{m}_{22}M_{m}+a^{m}_{23}N_{m}=(a^{0}_{22}+a^{m}_{33})M_{m}.$$
Therefore, $$a^{m}_{21}=0,\;\;\;
a^{m}_{22}=a^{0}_{22}+a^{m}_{33},\;\;\; a^{m}_{23}=0.$$ Then
$$a^{0}_{33}=0,\quad D(M_{m})=a^{m}_{22}M_{m}=(a^{0}_{22}+a^{m}_{33})M_{m}.$$
Since  $D[N_{0},L_{m}]=[D(N_{0}),L_{m} ]+[N_{0}, D(L_{m})]$, we
have
$$a^{0}_{31}mL_{m}+2a^{m}_{12}M_{m}=0.$$
Then $a^{0}_{31}m=0, 2a^{m}_{12}=0$ for all $m\in\Z$. Therefore,
$$a^{0}_{31}=0,\; \;  a^{m}_{12}=0,\quad   D(N_{0})=a^{0}_{32}M_{0}.$$
According to $D[N_{m},N_{n}]=[D(N_{m}),N_{n}]+[N_{m},D(N_{n})]$,
we obtain
$$(a^{m}_{31}n-a^{n}_{31}m)N_{m+n}+2(a^{n}_{32}-a^{m}_{32})M_{m+n}=0.$$
Then $a^{m}_{31}n=a^{n}_{31}m, a^{n}_{32}=a^{m}_{32}$ for all $m,
n\in\Z$. Hence, we have
$$a^{m}_{31}=a^{1}_{31}m, \;\;    a^{m}_{32}=a^{0}_{32}, \quad \forall\; m\in\Z.$$
By the following two relations
\begin{eqnarray*}
& &D[Y_{m+\frac{1}{2}},Y_{n+\frac{1}{2}}]
=[D(Y_{m+\frac{1}{2}}),Y_{n+\frac{1}{2}}]+[Y_{m+\frac{1}{2}},D(Y_{n+\frac{1}{2}})],
\\
& &D[L_{m},L_{n}]=[D(L_{m}),L_{n}]+[L_{m},D(L_{n})],
\end{eqnarray*}
we have
\begin{equation}\label{6.1}
a^{m+n+1}_{22}=b^{m+\frac{1}{2}}+b^{n+\frac{1}{2}},\quad\quad\;m\neq\;n.
\end{equation}
% $$(n-m)a^{m+n}_{11}L_{m+n}=[a^{m}_{11}L_{m},L_{n}]+[L_{m},a^{n}_{11}L_{n}]=(a^{m}_{11}+a^{n}_{11})(n-m)L_{m+n}.$$
\begin{equation}\label{6.2}
a^{m+n}_{11}=a^{m}_{11}+a^{n}_{11}, \quad\quad m\neq n.
\end{equation}
It is easy to see that $a^{-m}_{11}=-a^{m}_{11}$ for all $m\in\Z.$
Let $n=1$ in (\ref{6.2}), then
$$a^{m+1}_{11}=a^{m}_{11}+a^{1}_{11}, \quad\quad m\neq 1.$$
By induction on $m\in\Z^{+}$ and $m\geq 3$, we have
$$a^{m}_{11}=a^{2}_{11}+(m-2)a^{1}_{11}, \quad\quad m\geq 3.$$
Let $m=4, n=-2$ in (\ref{6.2}), then we have
$a^{2}_{11}=2a^{1}_{11}.$ Therefore,
$$a^{m}_{11}=ma^{1}_{11}, \quad\quad \forall\; m\in\Z.$$
Since $D[L_{m},M_{n}]=[D(L_{m}),M_{n}]+[L_{m},D(M_{n})]$ and
$$D[L_{m},N_{n}]=[D(L_{m}),N_{n}]+[L_{m},D(N_{n})],$$ we get
\begin{equation}\label{6.3}
a^{m+n}_{22}=ma^{1}_{11}+a^{n}_{22}, \quad\quad n\neq 0,
\end{equation}
$$a^{1}_{31}(m+n)=a^{1}_{31}(n-m),\;\;\;  a^{m+n}_{33}=ma^{1}_{11}+a^{n}_{33},\quad\;n\neq 0.$$
Then we can deduce that
$$a^{1}_{31}=0,\;\;\;  a^{m+1}_{33}=ma^{1}_{11}+a^{1}_{33}.$$
%$$a^{m}_{21}=0,\;\;\; a^{m}_{22}=a^{0}_{22}+a^{m}_{33},\;\;\;a^{m}_{23}=0.$$
Note $a^{0}_{33}=0$, then $a^{1}_{11}=a^{1}_{33}.$ Therefore,
$$a^{m}_{33}=ma^{1}_{11}, \;\;\;a^{m}_{22}=a^{0}_{22}+ma^{1}_{11}, \quad\quad \forall \; m\in\Z.$$
Because
$D[L_{m},Y_{n+\frac{1}{2}}]=[D(L_{m}),Y_{n+\frac{1}{2}}]+[L_{m},D(Y_{n+\frac{1}{2}})]$,
we have
$$b^{m+n+\frac{1}{2}}=ma^{1}_{11}+b^{n+\frac{1}{2}},\;\;\;  n+\frac{1-m}{2}\neq 0.$$
Let $m=1$, then
$$b^{n+1+\frac{1}{2}}=a^{1}_{11}+b^{n+\frac{1}{2}},\;\;\;  n\neq 0.$$
By (\ref{6.1}), we get
\begin{equation}\label{6.4}
a^{0}_{22}+(m+n+1)a^{1}_{11}=b^{m+\frac{1}{2}}+b^{n+\frac{1}{2}},\quad\quad\;m\neq\;n.
\end{equation}
Let $m=0,n=1$ in (\ref{6.4}), then we have
$$a^{0}_{22}+a^{1}_{11}=2b^{\frac{1}{2}}.$$
Let $n=0$ in (\ref{6.4}), then we obtain
$b^{m+\frac{1}{2}}=b^{\frac{1}{2}}+ ma^{1}_{11}$ for all $m\in\Z$.
Set $a^{1}_{11}=a, b^{\frac{1}{2}}=b, a^{0}_{32}=c$, then
$a^{0}_{22}=2b-a$ and
\begin{eqnarray*}
& &D(L_{m})=maL_{m},\quad\quad\quad \;D(M_{m})=(2b-a+ma)M_{m},
\\
& &D(N_{m})=cM_{m}+maN_{m},\; D(Y_{m+\frac{1}{2}})=(b+
ma)Y_{m+\frac{1}{2}},
\end{eqnarray*}
for all $m\in\Z$. Then we can deduce that
$$D=ad(aL_{0}-\dis\frac{c}{2}M_{0}+(b-\frac{a}{2})N_{0}).$$
\end{proof}

From the above lemmas, we obtain the following theorem.
\begin{theorem}\label{T3.1}
$Der({\tsv})=Inn({\tsv}),$ i.e., $\rm H^1(\tsv,\tsv)=0$.
\end{theorem}

\begin{lem}\label{T1.2}
 $C({\tsv})=\{0\}$, where  $C({\tsv})$ is the center of
${\tsv}$.
\end{lem}

\begin{proof}
  For any $E_{\frac{n}{2}}\in {\tsv}_{\frac{n}{2}}$,  we
have
$$[L_{0}, E_{\frac{n}{2}}]=\frac{n}{2}E_{\frac{n}{2}}.$$  It forces  $x\in
{\tsv}_{0}$,  for any $x\in C({\tsv})$, since $[L_{0}, x]=0$. Let
$x=aL_{0}+bN_{0}+cM_{0}$, where $a, b, c\in\C$. Then
$$[x, L_{1}]=[aL_{0}+bN_{0}+cM_{0}, L_{1}]=[aL_{0}, L_{1}]=aL_{1}=0.$$
So $a=0$. By the following relations,
$$[x,  Y_{\frac{1}{2}}]=[bN_{0}+cM_{0},Y_{\frac{1}{2}}]=[bN_{0},Y_{\frac{1}{2}}]=bY_{\frac{1}{2}}=0,$$
$$[x, N_{0}]=[cM_{0}, N_{0}]=-2cM_{0}=0,$$
we have $b=0$ and  $c=0$. Therefore, $x=0$.
\end{proof}

By Lemma \ref{T1.2} and Theorem \ref{T3.1}, we have
\begin{corollary}
${\tsv}$ is an infinite-dimensional complete Lie algebra.
\end{corollary}

%%%%%%%%%%%%%%%%%%%%%%%%%%%%%%%%%%%%%%%%%%%%%%%%%%%%%%%%%%%%%%%%%%%%%%%%%%%%%%%%%%%%%%%%%%%%%%%%%%%%%%
%
\section{\bf The Universal Central Extension  of ${\tsv}$}
%
%%%%%%%%%%%%%%%%%%%%%%%%%%%%%%%%%%%%%%%%%%%%%%%%%%%%%%%%%%%%%%%%%%%%%%%%%%%%%%%%%%%%%%%%%%%%%%%%%%%%%%
\label{sub3-c-related}

In this section, we discuss  the structure of  the universal
central extension of ${\tsv}$. Let us first recall some basic
concepts. Let $\mathfrak{g}$ be a Lie algebra. A bilinear function
$\psi: \mathfrak{g}\times \mathfrak{g}\longrightarrow \C$ is
called a 2-cocycle on $\mathfrak{g}$ if for all $x, y, z\in
\mathfrak{g}$, the following two conditions are satisfied:
\begin{equation}\nonumber
\psi(x, y)=-\psi(y, x),
\end{equation}
\begin{equation}\label{e3.0.1}
\psi([x, y], z)+\psi([y, z], x)+\psi([z, x], y)=0.
\end{equation}
 For any linear function $f:
\mathfrak{g}\longrightarrow \C$, one can define  a 2-cocycle
$\psi_{f}$ as follows
$$\psi_{f}(x, y)=f([x, y]), \quad\quad\quad \forall\; x, y\in \mathfrak{g}.$$
Such a 2-cocycle is called a 2-coboudary on $\mathfrak{g}$. Denote
by $C^{2}(\mathfrak{g}, \C)$ the vector space of 2-cocycles on
$\mathfrak{g}$, $B^{2}(\mathfrak{g}, \C)$ the vector space of
2-coboundaries on $\mathfrak{g}$. Then the quotient space
$H^{2}(\mathfrak{g},\C)=C^{2}(\mathfrak{g},\C)/B^{2}(\mathfrak{g},\C)$
is called the second cohomology group of $\mathfrak{g}$.

\begin{theorem}\label{T3.1}
$\dim H^{2}({\tsv}, \C)=3.$
\end{theorem}

\begin{proof}
Let $\vp: {\tsv}\times
{\tsv}\longrightarrow \C$ be a 2-cocycle on
${\tsv}$. Let
$f:{\tsv}\longrightarrow \C$ be a linear
function defined by
\begin{eqnarray*}
& &f(L_{0})=-\dis\frac{1}{2}\varphi(L_{1},L_{-1}), \quad
f(L_{m})=\frac{1}{m}\vp(L_{0}, L_{m}),\; m\neq 0,
\\
& &f(M_{0})=-\varphi(L_{1},M_{-1}),\quad
f(M_{m})=\frac{1}{m}\vp(L_{0}, M_{m}),\; m\neq 0,
\\
& &f(N_{0})=-\varphi(L_{1},N_{-1}), \quad
f(N_{m})=\frac{1}{m}\vp(L_{0}, N_{m}),\; m\neq 0,
\\
& &f(Y_{m+\frac{1}{2}})=\frac{1}{m+\frac{1}{2}}\vp(L_{0},
Y_{m+\frac{1}{2}}), \quad\quad\forall\; m\in\Z.
\end{eqnarray*}
%Then there exists a 2-coboundary $\vp_{f}$ satisfying
%$$\vp_{f}(x, y)=f([x, y]), \quad \forall\; x, y\in {\tsv}.$$
Let $\overline{\vp}=\vp-\vp_{f}$, where $\vp_{f}$ is the
2-coboundary induced by $f$, then
$$\overline{\vp}(x, y)=\vp(x, y)-f([x, y]), \quad \forall\; x, y\in {\tsv}.$$
%%%%%%%%%%%%%%%%%%%%%%%%%%%%%%%%%%%%%%%%%%%%%%%%%%%%%%%%%%%%%%%%%%%%%%%%%%%%%%%%%%%%%%%%%%%%%%%%%%%%%%
%
%%%%%%%%%%%%%%%%%%%%%%%%%%%%%%%%%%%%%%%%%%%%%%%%%%%%%%%%%%%%%%%%%%%%%%%%%%%%%%%%%%%%%%%%%%%%%%%%%%%%%%
By the known result on the central extension of the classical Witt
algebra ( see \cite{BM} or \cite{LJ} ), we have
$$\overline{\vp}(L_{m}, L_{n})=\a\delta_{m+n,0}(m^{3}-m), \quad \forall\; m,n\in \Z, \; \a\in\C.$$
 By the fact that
$\vp([L_{0}, L_{m}], M_{n})+\vp([L_{m}, M_{n}], L_{0})+\vp([M_{n},
L_{0}], L_{m} )=0$, we get
$$(m+n) \vp(L_{m},M_{n})=n\vp( L_{0}, M_{m+n}).$$
So
\begin{equation}\label{3.0.5}
\vp(L_{m},M_{n})=\frac{n}{m+n}\vp( L_{0}, M_{m+n}), \quad
m+n\neq0.
\end{equation}
Then it is easy to deduce that
$$\overline{\vp}(L_{m},M_{n})=0, \quad\quad\quad\;m+n\neq0.$$
Furthermore,
\begin{equation}\label{3.0.6}
\overline{\vp}(L_{m},M_{-m})=\vp(L_{m},M_{-m})+m
f(M_{0})=\vp(L_{m},M_{-m})-m \vp(L_{1},M_{-1}).
\end{equation}
%Denote $\overline{\vp}(L_{m},M_{-m})=b(m).$ Therefore,
%\begin{equation}\label{3.0.7}
%\overline{\vp}(L_{m},M_{n})=\delta_{m+n,0}b(m).
%\end{equation}
Similarly, denote $\overline{\vp}(L_{m},N_{-m})=c(m)$ for all
$m\in\Z$, then we have
\begin{equation}\label{3.0.8}
\overline{\vp}(L_{m},N_{n})=\delta_{m+n,0}c(m),\quad\overline{\vp}(L_{m},N_{-m})=\vp(L_{m},N_{-m})-m\vp(L_{1},N_{-1}).
\end{equation}
By (\ref{e3.0.1}), $\vp([L_{0}, L_{m}],
Y_{n+\frac{1}{2}})+\vp([L_{m}, Y_{n+\frac{1}{2}}],
L_{0})+\vp([Y_{n+\frac{1}{2}}, L_{0}], L_{m} )=0$, so
% $$m\vp(L_{m}, Y_{n+\frac{1}{2}})+(n+\frac{1-m}{2})\vp(Y_{m+n+\frac{1}{2}},L_{0}) +(n+\frac{1}{2})\vp(L_{m}, Y_{n+\frac{1}{2}})=0.$$
$$(m+n+\frac{1}{2})\vp(L_{m},Y_{n+\frac{1}{2}})=(n+\frac{1-m}{2})\vp(L_{0},Y_{m+n+\frac{1}{2}}).$$
Then for all $m, n\in\Z$, we get
$$\vp(L_{m},Y_{n+\frac{1}{2}})=\frac{n+\frac{1-m}{2}}{m+n+\frac{1}{2}}\vp(L_{0},Y_{m+n+\frac{1}{2}}).$$
Consequently, we have
$$\overline{\vp}(L_{m},Y_{n+\frac{1}{2}})=0, \quad\quad\quad\;m,n\in\Z.$$
%%%%%%%%%%%%%%%%%%%%%%%%%%%%%%%%%%%%%%%%%%%%%%%%%%%%%%%%%%%%%%%%%%%%%%%%%%%%%%%%%%%%%%%%%%%%%
From the relation that $$\vp([N_{0}, Y_{m+\frac{1}{2}}],
Y_{-m-\frac{1}{2}})+\vp([Y_{m+\frac{1}{2}}, Y_{-m-\frac{1}{2}}],
N_{0}) +\vp([Y_{-m-\frac{1}{2}}, N_{0}], Y_{m+\frac{1}{2}})=0,\;\;\;
m\in\Z,$$ we have
\begin{equation}\label{ext1}
\vp(Y_{m+\frac{1}{2}},Y_{-m-\frac{1}{2}})=(m+\frac{1}{2})\vp(N_{0},
M_{0}),\;\;\; m\in\Z.
\end{equation}
By (\ref{e3.0.1}), $\vp([L_{m}, Y_{p+\frac{1}{2}}],
Y_{q+\frac{1}{2}})+\vp([Y_{p+\frac{1}{2}}, Y_{q+\frac{1}{2}}],
L_{m})+\vp([Y_{q+\frac{1}{2}}, L_{m}], Y_{p+\frac{1}{2}} )=0$, then
$$(p+\frac{1-m}{2})\vp( Y_{m+p+\frac{1}{2}},Y_{q+\frac{1}{2}})+(p-q)\vp(M_{p+q+1},L_{m})
-(q+\frac{1-m}{2})\vp(Y_{m+q+\frac{1}{2}}, Y_{p+\frac{1}{2}} )=0,$$
for all $m, p, q,\in\Z$. Let $m=-p-q-1$, then for all $p, q\in\Z$,
we have
\begin{equation}\label{3.0.25}
(p+1+\frac{p+q}{2})\vp(Y_{-q-\frac{1}{2}},Y_{q+\frac{1}{2}})+(p-q)\vp(M_{p+q+1},L_{-p-q-1})
-(q+1+\frac{p+q}{2})\vp(Y_{-p-\frac{1}{2}}, Y_{p+\frac{1}{2}} )=0.
\end{equation}
Using (\ref{ext1}), we get
\begin{equation}\label{ext2}
\vp(L_{-p-q-1}, M_{p+q+1})=\frac{p+q+1}{2}\vp(N_{0}, M_{0}),
\quad\quad\quad\; p\neq q.
\end{equation}
Letting $p=-2, q=0$ in (\ref{ext2}),  we have
\begin{equation}\label{3.0.27}
\vp(L_{1}, M_{-1})=-\frac{1}{2}\vp(N_{0},
M_{0})=-\vp(Y_{\frac{1}{2}},Y_{-\frac{1}{2}}).
\end{equation}
It follows from (\ref{ext1}) and  (\ref{ext2}) that
\begin{equation}\label{ext3}
\vp(Y_{m+\frac{1}{2}},Y_{-m-\frac{1}{2}})=-(2m+1)\vp(L_{1}, M_{-1}),
\;\;\; m\in\Z,
\end{equation}
\begin{equation}\label{3.0.26}
\vp(L_{m},
M_{-m})=-m\vp(Y_{\frac{1}{2}},Y_{-\frac{1}{2}})=m\vp(L_{1}, M_{-1}),
\quad\;\;\; m\in\Z.
\end{equation}
By (\ref{3.0.6}) and (\ref{3.0.26}), we have
$\overline{\vp}(L_{m},M_{-m})=0$ for all $m\in\Z$. Therefore,
$$\overline{\vp}(L_{m},M_{n})=0, \quad\; \forall\; m, n\in\Z.$$
%%%%%%%%%%%%%%%%%%%%%%%%%%%%%%%%%%%%%%%%%%%%%%%%%%%%%%%%%%%%%%%%%%%%%%%%%%%%%%%%%%%%%%%%%
 According to (\ref{e3.0.1}),  $\vp([L_{m}, L_{n}],
N_{-m-n})+\vp([L_{n}, N_{-m-n}], L_{m})+\vp([N_{-m-n}, L_{m}],
L_{n} )=0$, then we have
\begin{equation}\nonumber
(n-m)\vp(L_{m+n},N_{-m-n})-(m+n)\vp(N_{-m},L_{m})+(m+n)\vp(N_{-n},L_{n})=0.
\end{equation}
By (\ref{3.0.8}), we get
\begin{equation}\label{3.0.11}
(n-m)c(m+n)+(m+n)c(m)=(m+n)c(n).
\end{equation}
Let $n=1$ in (\ref{3.0.11}) and note $c(0)=c(1)=0$, then we obtain
\begin{equation}\label{3.0.12}
(m-1)c(m+1)=(m+1)c(m).
\end{equation}
Let $n=1-m$ in (\ref{3.0.11}), then
\begin{equation}\nonumber
c(m)=c(1-m), \quad\quad \forall\; m\in\Z.
\end{equation}
By induction on $m$, we  deduce that $c(-1)=c(2)$ determines all
$c(m)$ for $m\in\Z$. On the other hand, $c(m)=m^{2}-m$ is a
solution of equation (\ref{3.0.12}). So
$$c(m)=\b(m^{2}-m), \quad  \b\in\C,$$ is the general
solution of equation (\ref{3.0.12}). Therefore,
\begin{equation}\label{3.0.14}
\overline{\vp}(L_{m}, N_{n})=\delta_{m+n,0}\b(m^{2}-m), \quad\;
\forall\; m, n\in\Z, \; \b\in\C.
\end{equation}
 By (\ref{e3.0.1}), we have
$$\vp([M_{p},Y_{m+\frac{1}{2}}],Y_{n+\frac{1}{2}})+\vp([Y_{m+\frac{1}{2}},
Y_{n+\frac{1}{2}}], M_{p})+\vp([Y_{n+\frac{1}{2}},
M_{p}],Y_{m+\frac{1}{2}} )=0.$$ Then $(m-n)\vp(M_{m+n+1}, M_{p})=0.$
Let $m+n+1=q$, then we have $(2m+1-q)\vp(M_{q}, M_{p})=0$. This
means that
$$\vp(M_{p}, M_{q})=0, \quad\forall\; p, q\in\Z.$$ Therefore, we
have
\begin{equation}\label{3.1.14}
\overline{\vp}(M_{m}, M_{n})=\vp(M_{m},M_{n})=0, \quad\forall\;
m,n\in\Z.
\end{equation}
By (\ref{e3.0.1}), $\vp([L_{0}, M_{m}], N_{n})+\vp([M_{m}, N_{n}],
L_{0})+\vp([N_{n}, L_{0}], M_{m} )=0$, we have
\begin{equation}\label{3.1.16}
(m+n)\vp( M_{m}, N_{n})=-2\vp(L_{0},M_{m+n}).
\end{equation}
Then for $m+n\neq 0$, we have
$$\overline{\vp}(M_{m}, N_{n})=\vp(M_{m}, N_{n})+\frac{2}{m+n}\vp(L_{0}, M_{m+n})=0.$$
On the other hand,
$$\overline{\vp}(M_{m},N_{-m})=\vp(M_{m},N_{-m})+2f(M_{0})=\vp(M_{m},N_{-m})-2\vp(L_{1},M_{-1}).$$
By (\ref{e3.0.1}),
$\vp([N_{p},Y_{m+\frac{1}{2}}],Y_{n+\frac{1}{2}})+\vp([Y_{m+\frac{1}{2}},
Y_{n+\frac{1}{2}}],
N_{p})+\vp([Y_{n+\frac{1}{2}},N_{p}],Y_{m+\frac{1}{2}} )=0$, we
have
$$\vp(Y_{p+m+\frac{1}{2}},Y_{n+\frac{1}{2}})+(m-n)\vp(M_{m+n+1}, N_{p})-\vp(Y_{p+n+\frac{1}{2}},Y_{m+\frac{1}{2}} )=0.$$
Let $n=-m-p-1$, then we have
$$\vp(Y_{p+m+\frac{1}{2}},Y_{-p-m-\frac{1}{2}})+(2m+p+1)\vp(M_{-p}, N_{p})-\vp(Y_{-m-\frac{1}{2}},Y_{m+\frac{1}{2}} )=0.$$
By (\ref{ext3}), we get
$$\vp(M_{-p}, N_{p})=2\vp(L_{1},M_{-1}), \quad \forall\; p\in\Z.$$
Therefore,
\begin{equation}\label{3.1.17}
\overline{\vp}(M_{m}, N_{n})=0, \quad\forall\; m,n\in\Z.
\end{equation}
By (\ref{e3.0.1}), $\vp([L_{0}, M_{m}],
Y_{n+\frac{1}{2}})+\vp([M_{m}, Y_{n+\frac{1}{2}}],
L_{0})+\vp([Y_{n+\frac{1}{2}}, L_{0}], M_{m} )=0$, we have
$$(m+n+\frac{1}{2})\vp( M_{m},Y_{n+\frac{1}{2}})=0.$$
So $\vp( M_{m},Y_{n+\frac{1}{2}})=0$ for all $ m, n\in\Z.$
Therefore
$$\overline{\vp}(M_{m},Y_{n+\frac{1}{2}})=\vp(M_{m},Y_{n+\frac{1}{2}})-f([M_{m},Y_{n+\frac{1}{2}}])=0.$$
By (\ref{e3.0.1}), $\vp([L_{0}, N_{p}],N_{q})+\vp([N_{p}, N_{q}],
L_{0})+\vp([N_{q}, L_{0}], N_{p})=0$, we have
% $$p\vp(M_{p},M_{q})-q\vp(M_{q},M_{p})=0$$
$$(p+q)\vp(N_{p},N_{q})=0.$$
So
$\overline{\vp}(N_{p},N_{q})=\vp(N_{p},N_{q})=\delta_{p+q,0}k(p)$,
where $\overline{\vp}(N_{p},N_{-p})=k(p)$.
% i.e.,
% \begin{equation}\label{3.1.13}
% \overline{\vp}(N_{p},N_{q})=\vp(N_{p},N_{q})=\delta_{m+n,0}k(p).
% \end{equation}
Then by (\ref{e3.0.1}), $\vp([L_{-p-q}, N_{p}],N_{q})+\vp([N_{p},
N_{q}], L_{-p-q})+\vp([N_{q}, L_{-p-q}], N_{p})=0$, we get
 $$pk(q)=qk(p).$$ Let $q=1$, then $k(p)=pk(1).$ Set
$k(1)=k$, then we have
\begin{equation}\nonumber
\overline{\vp}(N_{m}, N_{n})=k\delta_{m+n,0}m, \quad\forall\;
m,n\in\Z.
\end{equation}
By (\ref{e3.0.1}), $\vp([L_{0},
N_{m}],Y_{n+\frac{1}{2}})+\vp([N_{m}, Y_{n+\frac{1}{2}}],
L_{0})+\vp([Y_{n+\frac{1}{2}}, L_{0}], N_{m} )=0$, we have
\begin{equation}\nonumber
(m+n+\frac{1}{2})\vp(N_{m},Y_{n+\frac{1}{2}})=\vp(L_{0},Y_{m+n+\frac{1}{2}}).
\end{equation}
Then
\begin{eqnarray*}
\overline{\vp}(N_{m},Y_{n+\frac{1}{2}})&=&\vp(N_{m},Y_{n+\frac{1}{2}})-f(Y_{m+n+\frac{1}{2}})
\\
&=&\vp(N_{m},Y_{n+\frac{1}{2}})-\frac{1}{m+n+\frac{1}{2}}\vp(L_{0},Y_{m+n+\frac{1}{2}})
\\
&=&0,\quad \forall\; m,n\in\Z.
\end{eqnarray*}
By (\ref{e3.0.1}),
$\vp([L_{0},Y_{m+\frac{1}{2}}],Y_{n+\frac{1}{2}})+\vp([Y_{m+\frac{1}{2}},
Y_{n+\frac{1}{2}}], L_{0})+\vp([Y_{n+\frac{1}{2}},
L_{0}],Y_{m+\frac{1}{2}} )=0$, we have
\begin{equation}\label{3.1.11}
(m+n+1)\vp(Y_{m+\frac{1}{2}},Y_{n+\frac{1}{2}})
=(m-n)\vp(L_{0},M_{m+n+1}).
\end{equation}
For $m+n+1\neq 0$,
\begin{eqnarray*}
\overline{\vp}(Y_{m+\frac{1}{2}},Y_{n+\frac{1}{2}})
&=&\vp(Y_{m+\frac{1}{2}},Y_{n+\frac{1}{2}})-f([Y_{m+\frac{1}{2}},Y_{n+\frac{1}{2}}])
\\
&=&\vp(Y_{m+\frac{1}{2}},Y_{n+\frac{1}{2}})-(m-n)f(M_{m+n+1})
\\
&=&\vp(Y_{m+\frac{1}{2}},Y_{n+\frac{1}{2}})-\frac{m-n}{m+n+1}\vp(L_{0},M_{m+n+1}).
\end{eqnarray*}
By (\ref{3.1.11}), we have
$$\overline{\vp}(Y_{m+\frac{1}{2}},Y_{n+\frac{1}{2}})=0, \quad m+n+1\neq 0.$$
On the other hand, for all $m\in\Z$,
\begin{eqnarray*}
\overline{\vp}(Y_{m+\frac{1}{2}},Y_{-m-\frac{1}{2}})
&=&\vp(Y_{m+\frac{1}{2}},Y_{-m-\frac{1}{2}})-f([Y_{m+\frac{1}{2}},Y_{-m-\frac{1}{2}}])
\\
&=&\vp(Y_{m+\frac{1}{2}},Y_{-m-\frac{1}{2}})-(2m+1)f(M_{0})
\\
&=&\vp(Y_{m+\frac{1}{2}},Y_{-m-\frac{1}{2}})+(2m+1)\vp(L_{1},M_{-1}).
\end{eqnarray*}
By (\ref{ext3}), we have
$$\overline{\vp}(Y_{m+\frac{1}{2}},Y_{-m-\frac{1}{2}})=0, \quad\quad\forall\; m\in\Z.$$
So
$$\overline{\vp}(Y_{m+\frac{1}{2}},Y_{n+\frac{1}{2}})=0, \quad
\forall\; m, n\in\Z.$$ Therefore, $\overline{\vp}$ is determined
by the following three  2-cocycles
\begin{eqnarray*}
& &\vp_{1}(L_{m}, L_{n})=\delta_{m+n,0}\dis\frac{m^{3}-m}{12},
\\
& &\vp_{2}(L_{m}, N_{n})=\delta_{m+n,0}(m^{2}-m),
\\
& &\vp_{3}(N_{m}, N_{n})=\delta_{m+n,0}n.
\end{eqnarray*}
%%%%%%%%%%%%%%%%%%%%%%%%%%%%%%%%%%%%%%%%%%%%%%%%%%%%%%%%%%%%%%%%%%%%%%%%%%%%%%%%%%%%%%%%%%

\end{proof}
\begin{remark}
 It is referred in \cite{U} that $\tsv$ has three independent
classes of central extensions given by the cocycles $$
c_{1}(L_{m}, L_{n})=\delta_{m+n,0}\dis\frac{m^{3}-m}{12},\;
 c_{2}(L_{m}, N_{n})=\delta_{m+n,0}m^{2},\; c_{3}(N_{m}, N_{n})=\delta_{m+n,0}m.$$
Here we prove in detail that $\tsv$ has only three independent
classes of central extensions.
\end{remark}

Let $\mathfrak{g}$ be a perfect Lie algebra, i.e., $[\mathfrak{g},
\mathfrak{g}]=\mathfrak{g}$. $(\widehat{\mathfrak{g}}, \pi)$ is
called a central extension of $\mathfrak{g}$ if $\pi:
\widehat{\mathfrak{g}}\longrightarrow \mathfrak{g}$ is a surjective
homomorphism whose kernel lies in the center of the Lie algebra
$\widehat{\mathfrak{g}}$. The pair $(\widehat{\mathfrak{g}}, \pi)$
is called a covering of $\mathfrak{g}$ if  $\widehat{\mathfrak{g}}$
is perfect. A covering $(\widehat{\mathfrak{g}}, \pi)$ is a
universal central extension of $\mathfrak{g}$ if for every central
extension $(\widehat{\mathfrak{g}}', \varphi)$ of $\mathfrak{g}$
there is a unique homomorphism $\psi:
\widehat{\mathfrak{g}}\longrightarrow\widehat{\mathfrak{g}}'$ for
which $\varphi\psi=\pi$. In \cite{G}, it is proved that every
perfect Lie algebra has a universal central extension.

Let
${\hsv}={{\tsv}}\bigoplus\C\;C_{L}\bigoplus\C\;C_{LN}\bigoplus\C\;C_{N}$
be a vector space over the complex field $\C$ with a basis
$\{L_{n}, M_{n}, N_{n}, Y_{n+\frac{1}{2}}, C_{L}, C_{LN}, C_{N} \
|\ n\in\Z \}$ satisfying the following relations
$$[L_{m}, L_{n}]=(n-m)L_{m+n}+\delta_{m+n,0}\dis\frac{m^{3}-m}{12}C_{L},$$
$$[N_{m}, N_{n}]=n\delta_{m+n,0}C_{N},$$
$$[L_{m}, N_{n}]=nN_{m+n}+\delta_{m+n,0}(m^{2}-m)C_{LN},  $$
$$[M_{m}, M_{n}]=0, \ \ [Y_{m+\frac{1}{2}}, Y_{n+\frac{1}{2}}]=(m-n)M_{m+n+1},$$
$$[L_{m}, M_{n}]=nM_{m+n},  \ \ [L_{m},Y_{n+\frac{1}{2}}]=(n+\dis\frac{1-m}{2})Y_{m+n+\frac{1}{2}}, $$
$$[N_{m}, M_{n}]=2M_{m+n},  \ \ [N_{m},Y_{n+\frac{1}{2}}]=Y_{m+n+\frac{1}{2}}, \ \ [M_{m},Y_{n+\frac{1}{2}}]=0,$$
$$[{\hsv}, C_{L}]=[{\hsv}, C_{LN}]=[{\hsv}, C_{N}]=0,$$
for all $m,n\in\Z$.   Denote
$$H=\bigoplus\limits_{n\in\Z}\C\;N_{n}\bigoplus\C\;C_{N},\;
\mathfrak{Vir}=\bigoplus\limits_{n\in\Z}\C\;L_{n}\bigoplus\C\;C_{L},\;
H_{Vir}=H\bigoplus\;\mathfrak{Vir}\bigoplus\C\;C_{LN},$$
$$\mathcal{S}=\bigoplus\limits_{n\in\Z}\C\;M_{n}\bigoplus\limits_{n\in\Z}\C\;Y_{n+\frac{1}{2}},
 \quad \mathcal{H_{S}}=H\bigoplus \mathcal{S}.$$ They are all Lie subalgebras
of ${\hsv}$, where $H$ is an infinite-dimensional Heisenberg
algebra, $\mathfrak{Vir}$ is the classical Virasoro algebra,
$H_{Vir}$ is the twisted Heisenberg-Virasoro algebra,
$\mathcal{S}$ is a two-step nilpotent Lie algebra and
$\mathcal{H_{S}}$ is the semi-direct product of the Heisenberg
algebra $H$ and $\mathcal{S}$. Then ${\hsv}$ is the semi-direct
product of the twisted Heisenberg-Virasoro algebra $H_{Vir}$ and
$\mathcal{S}$, where $\mathcal{S}$ is an ideal of ${\hsv}$ .

\begin{corollary}\label{C3.2}
 ${\hsv}$ is the
universal covering algebra of the extended
Schr\"{o}dinger-Virasoro algebra ${{\tsv}}$.
\qed
\end{corollary}

 Set $deg(C_{L})=deg(C_{LN})=deg(C_{N})=0.$ Then there is  a
$\dis\frac{1}{2}\Z$-grading on ${\hsv}$ by
$${\hsv}=\bigoplus\limits_{n\in\Z}{\hsv}_{\frac{n}{2}}
=(\bigoplus\limits_{n\in\Z}{\hsv}_{n})
\bigoplus(\bigoplus\limits_{n\in\Z}{\hsv}_{n+\frac{1}{2}}),$$
where
$${\hsv}_{n}=span\{L_{n}, M_{n}, N_{n}\},\;\;
n\in\Z\setminus \{0\} ,$$
$${\hsv}_{0}=span\{L_{0}, M_{0}, N_{0}, C_{L}, C_{LN}, C_{N}\},$$
$${\hsv}_{n+\frac{1}{2}}=span\{Y_{n+\frac{1}{2}}\}, \quad\quad \forall\;n\in\Z.$$

\begin{lem}\label{Theorem 2.2 in BM}{\rm{(cf. \cite{BM})}}
If  $\mathfrak{g}$  is a perfect Lie algebra and
$\widehat{\mathfrak{g}}$ is a universal central extension of
$\mathfrak{g}$, then every derivation of $\mathfrak{g}$ lifts to a
derivation of $\widehat{\mathfrak{g}}$. If $\mathfrak{g}$ is
centerless, the lift is unique and
$Der(\widehat{\mathfrak{g}})\cong Der(\mathfrak{g})$.
\qed
\end{lem}

It follows from Lemma \ref{Theorem 2.2 in BM} and Corollary
\ref{C3.2} that
$$D(C_{L})=D(C_{LN})=D(C_{N})=0,$$ for all $D\in
{\hsv}$. Therefore,
$Der({\hsv})=Inn({\hsv}).$

%%%%%%%%%%%%%%%%%%%%%%%%%%%%%%%%%%%%%%%%%%%%%%%%%%%%%%%%%%%%%%%%%%%%%%%%%%%%%%%%%%%%%%%%%%%%%%%%%%%%%%%%%%%%%%%%%
%
\section{\bf The  Universal Central Extension of ${{\tsv}}$ in the Category of Leibniz Algebras }
%
%%%%%%%%%%%%%%%%%%%%%%%%%%%%%%%%%%%%%%%%%%%%%%%%%%%%%%%%%%%%%%%%%%%%%%%%%%%%%%%%%%%%%%%%%%%%%%%%%%%%%%%%%%%%%%%%%

 The concept of  Leibniz algebra was first introduced by Jean-Louis
Loday in \cite{LP} in his study of the so-called Leibniz homology
as a noncommutative analog of Lie algebra homology. A vector space
$\mathcal{L}$ equipped with a $\C$-bilinear map $[-,-]:
\mathcal{L}\times \mathcal{L}\to \mathcal{L}$  is called a Leibniz
algebra
 if the following Leibniz identity satisfies
\begin{equation}\label{L-J}
[x, [y, z]]= [[x, y], z]-[[x, z], y],\quad \forall\; x, y, z\in
\mathcal{L}.
\end{equation}
Lie algebras are definitely Leibniz algebras. A Leibniz algebra
$\mathcal{L}$ is a Lie algebra if and only if $[x, x]=0$ for all
$x\in \mathcal{L}$.

In \cite{LP}, Jean-Louis Loday and Teimuraz Pirashvili established
the concept of universal enveloping algebras of Leibniz algebras
and interpreted the Leibniz (co)homology $HL_*$ (resp. $HL^*$) as
a Tor-functor (resp. Ext-functor). A bilinear $\C$-valued form
$\psi$ on $\mathcal{L}$ is  called a Leibniz $2$-cocycle if
\begin{equation}\label{L-Co}
\psi(x, [y, z])= \psi([x, y], z)-\psi([x, z], y),\quad
\forall\;x,y, z\in \mathcal{L}.
\end{equation}
Similar to the 2-cocycle on  Lie algebras, a linear function $f$
on $\mathcal{L}$ can induce a Leibniz $2$-cocycle $\psi_f$, that
is,
$$\psi_f(x, y)=f([x, y]),\quad\quad \forall\ x, \ y\in
\mathcal{L}.$$ Such a Leibniz $2$-cocycle is called trival. The
one-dimensional Leibniz central extension corresponding to a
trivial Leibniz $2$-cocycle is also trivial.

In this section, we consider the universal central extension of
the extended Schr\"{o}dinger-Virasoro algebra ${\tsv}$ in the
category of Leibniz algebras.

Let $\mathfrak{g}$ be a Lie algebra. A bilinear form
$f:\mathfrak{g}\longrightarrow \C$  is called invariant if
\begin{equation}\label{1.1}
f([x, y], z)=f(x, [y, z]), \quad \forall\; x,y,z\in \mathfrak{g}.
\end{equation}

\begin{proposition}\label{P1.1}
 There is no non-trivial invariant bilinear form on
 ${\tsv}$.
\end{proposition}

\begin{proof}
Let $f: {\tsv} \times
{\tsv} \longrightarrow \C$ be an invariant
bilinear form on ${\tsv}$.

 (1)\; For $m\neq 0$, we have
\begin{equation}\nonumber
f(L_{m}, L_{n})=-\frac{1}{m}f([L_{m},L_{0}], L_{n})
=-\frac{1}{m}f(L_{m},[L_{0}, L_{n}])=-\frac{n}{m}f(L_{m},L_{n}),
\end{equation}
% \begin{equation}\nonumber
% f(L_{m}, N_{n})=-\frac{1}{m}f([L_{m},L_{0}], N_{n})
% =-\frac{1}{m}f(L_{m},[L_{0}, N_{n}])=-\frac{n}{m}f(L_{m},N_{n}),
% \end{equation}
% \begin{equation}\nonumber
% f(N_{m}, N_{n})=-\frac{1}{m}f([N_{m},L_{0}], N_{n})
% =-\frac{1}{m}f(N_{m},[L_{0}, N_{n}])=-\frac{n}{m}f(N_{m},N_{n}).
% \end{equation}
So
% \begin{equation}\nonumber
% f(L_{m},L_{n})=0, \quad f(L_{m},N_{n})=0,    \quad
% f(N_{m},N_{n})=0, \quad\quad     m+n\neq 0, m\neq 0.
% \end{equation}
$$f(L_{m},L_{n})=0,\quad\quad     m+n\neq 0, m\neq 0.$$
Similarly, we have $$f(L_{m},N_{n})=0, \quad f(N_{m},N_{n})=0,
\quad\quad m+n\neq 0, m\neq 0.$$ For $n\neq 0$, we have
$$f(L_{0}, L_{n})=\frac{1}{n}f(L_{0},[L_{0}, L_{n}])=\frac{1}{n}f([L_{0},L_{0}], L_{n})=0,$$
% $$f(L_{0}, N_{n})=\frac{1}{n}f(L_{0},[L_{0}, N_{n}])=\frac{1}{n}f([L_{0},L_{0}], N_{n})=0,$$
% $$f(N_{0}, N_{n})=\frac{1}{n}f(N_{0},[L_{0}, N_{n}])=\frac{1}{n}f([N_{0},L_{0}], N_{n})=0,$$
$$f(L_{-n}, L_{n})=\frac{1}{3n}f([L_{-2n},L_{n}], L_{n})=\frac{1}{3n}f(L_{-2n},[L_{n}, L_{n}])=0,$$
% $$f(L_{-n}, N_{n})=\frac{1}{3n}f(L_{-n},[L_{-n}, N_{2n}])=\frac{1}{2n}f([L_{-n},L_{-n}], N_{2n})=0,$$
% $$f(N_{-n}, N_{n})=\frac{1}{n}f(N_{-n},[N_{-n}, L_{2n}])=\frac{1}{n}f([N_{-n},N_{-n}], L_{2n})=0.$$
Similarly, we can get
$$f(L_{0}, N_{n})=0,\; f(N_{0}, N_{n})=0,\; f(L_{-n}, N_{n})=0, \; f(N_{-n}, N_{n})=0.$$
On the other hand,
$$f(L_{0}, L_{0})=\frac{1}{2}f([L_{-1},L_{1}], L_{0})=\frac{1}{2}f(L_{-1},[L_{1},L_{0}])=-\frac{1}{2}f(L_{-1},L_{1})=0.$$
$$f(L_{0}, N_{0})=\frac{1}{2}f([L_{-1},L_{1}], N_{0})=\frac{1}{2}f(L_{-1},[L_{1},N_{0}])=0.$$
$$f(N_{0}, N_{0})=f([L_{-1},N_{1}], N_{0})=f(L_{-1},[N_{1},N_{0}])=0.$$
Therefore,
$$f(L_{m}, N_{n})=0, \quad  f(N_{m}, N_{n})=0, \quad  f(L_{m}, L_{n})=0, \quad\quad \forall\; m,n\in\Z.$$
Similarly,  we obtain
$$f(L_{m}, M_{n})=0, \quad f(M_{m}, M_{n})=0,   \quad \forall\; m,n\in\Z.$$

(2)\; For all $m,n\in\Z$, we have
$$f(L_{m},
Y_{n+\frac{1}{2}})=\frac{1}{n+\frac{1}{2}}f(L_{m},[L_{0},Y_{n+\frac{1}{2}}])
=\frac{1}{n+\frac{1}{2}}f([L_{m},L_{0}],Y_{n+\frac{1}{2}})
=-\frac{m}{n+\frac{1}{2}}f(L_{m},Y_{n+\frac{1}{2}}).$$ Then
$\dis\frac{m+n+\frac{1}{2}}{n+\frac{1}{2}}f(L_{m},Y_{n+\frac{1}{2}})=0.$
Obviously, $f(L_{m},Y_{n+\frac{1}{2}})=0$ for all $ m,n\in\Z$.

(3)\; For all $m,n\in\Z$, we have
$$f(N_{m}, M_{n})=\frac{1}{2}f(N_{m},[N_{0},
M_{n}])=\frac{1}{2}f([N_{m}, N_{0}], M_{n})=0,$$
$$f(N_{m},Y_{n+\frac{1}{2}})=f(N_{m},[N_{0},Y_{n+\frac{1}{2}}])=f([N_{m},N_{0}],Y_{n+\frac{1}{2}})=0,$$
$$f(M_{m},Y_{n+\frac{1}{2}})=\frac{1}{n+\frac{1}{2}}f(M_{m},[L_{0},Y_{n+\frac{1}{2}}])
=-\frac{1}{n+\frac{1}{2}}f([M_{m},Y_{n+\frac{1}{2}}],L_{0})=0,$$
\begin{eqnarray*}
f(Y_{m+\frac{1}{2}},Y_{n+\frac{1}{2}})&=&\frac{1}{m+\frac{1}{2}}f([L_{0},Y_{m+\frac{1}{2}}],Y_{n+\frac{1}{2}})
=\frac{1}{m+\frac{1}{2}}f(L_{0},[Y_{m+\frac{1}{2}},Y_{n+\frac{1}{2}}])
\\
&=&\frac{m-n}{m+\frac{1}{2}}f(L_{0},M_{m+n+1})=0.
\end{eqnarray*}
\end{proof}

\begin{remark}\label{R1.1}
 In fact, it is enough to check
that Proposition \ref{P1.1} holds for the set of generators
$\{L_{-2}, L_{-1},L_{1}, L_{2}, N_{1}, Y_{\frac{1}{2}}\}$. By the
proof of Proposition \ref{P1.1}, we can see that the process of the
computation  is independent of the symmetry of the bilinear form.
Similar to the method in section 4 in \cite{HPL}, we can deduce that
$$  HL^{2}({\tsv}, \C) = H^{2}({\tsv}, \C),$$
where $HL^{2}({\tsv}, \C)$ is the second Leibniz
cohomology group of ${\tsv}$. That is to say,
the universal central extension of ${\tsv}$ in
the category of Leibniz algebras is the same as that in the
category of Lie algebras.
\end{remark}

%\newpage

%%%%%%%%%%%%%%%%%%%%%%%%%%%%%%%%%%%%%%%%%%%%%%%%%%%%%%%%%%%%%%%%
%
\section{\bf  The Automorphism Group  of  ${\tsv}$}
%
%%%%%%%%%%%%%%%%%%%%%%%%%%%%%%%%%%%%%%%%%%%%%%%%%%%%%%%%%%%%%%%%
\label{sub5-c-related}

Denote by $Aut({\tsv})$ and ${\mathcal{I}}$ the  automorphism
group  and the inner automorphism group of ${\tsv}$ respectively.
Obviously, ${\mathcal{I}}$  is generated by $\exp(k {\rm{ad}}
M_{m}+l {\rm{ad}} Y_{n+\frac{1}{2}})$,  $m,n\in\Z,$  $k, l\in\C.$

For convenience, denote
$$L=span\{L_{n}\; |n\in\Z\},\;\;\;N=span\{N_{n}\;|n\in\Z\},$$
$$M=span\{M_{n}\;|n\in\Z\},\;\;\;Y=span\{Y_{n+\frac{1}{2}}\; |n\in\Z\}.$$

\begin{lem}\label{L5.2}
Let $\si\in Aut({\tsv})$, then
$$\si(M_{n})\in M,\quad \si(Y_{n+\frac{1}{2}})\in M+Y,\quad \si(N_{n})\in M+Y+N,$$
for all $n\in\Z$. In particular,
$$\si(N_{0})=\sum\limits_{i=p}^{q}
a_{i}M_{i}+N_{0}+\sum\limits_{j=s}^{t} b_{j}Y_{j+\frac{1}{2}},$$
for some $a_{i}, b_{j}\in\C$ and $p,q, s,t\in\Z.$
\end{lem}

\begin{proof}
Let $I$ be a nontrivial ideal of ${\tsv}$. Then $I$ is a
$L_{0}$-module. Since the decomposition of eigenvalue subspace of
$L_{0}$ is in concordance  with the $\dis\frac{1}{2}\Z$-grading of
${\tsv}$, we have
$$I=\bigoplus\limits_{n\in\Z}I_{\frac{n}{2}}
=\bigoplus\limits_{n\in\Z}I\bigcap{\tsv}_{\frac{n}{2}}.$$ Hence,
there exists some $n\in\Z$ such that $aL_{n}+bM_{n}+cN_{n}\in I$ or
$Y_{n+\frac{1}{2}}\in I$, where $a,b,c\in\C$ and not all zero. If
$aL_{n}+bM_{n}+cN_{n}\in I$, then
$$[aL_{n}+bM_{n}+cN_{n},M_{0}]=2cM_{n}\in I,\;\; [aL_{n}+bM_{n}+cN_{n},N_{0}]=-2bM_{n}\in I.$$
 If $b=c=0$, then $a\neq 0$ and $L_{n}\in\I$.
 But $[L_{n}, {\tsv}]={\tsv}$ for any $n\in\Z$, then we
 have $I={\tsv}$, a contradiction.
So $b\neq 0$ or $c\neq 0$. Therefore, $M_{n}\in I$, and  we  have
$aL_{n}+cN_{n}\in I$. Since $[aL_{n}+cN_{n}, N_{1}]=aN_{n+1}\in
I$, we get $N_{n+1}\in I$ if $a\neq 0$.

(1) If there exists some $M_{n}\in I$, by the fact that $[N_{m-n},
M_{n}]=2M_{m}$ for all $m\in\Z$, we obtain $M\subseteq I$.

(2) If there exists some $N_{n}\in I$ and $n\neq 0$, then
$N\subseteq I$ since $[L_{m-n}, N_{n}]=nN_{m}\in I$ for all
$m\in\Z$. On the other hand,  we have
$$[N_{0},M_{m}]=2M_{m}, \;\; [N_{0},Y_{m+\frac{1}{2}}]=Y_{m+\frac{1}{2}},$$
for all $m\in\Z$. So  $M\subseteq I,\; Y\subseteq I$, and
therefore $N\bigoplus M\bigoplus Y\subseteq I$.

If $N_{0}\in I$, according to the proof above, we have $M\subseteq
I, Y\subseteq I$. In addition, $[L, N_{0}]=0$ and $[N, N_{0}]=0$,
so $\C N_{0}\bigoplus M\bigoplus Y\subseteq I$.

(3) If there exists some $Y_{n+\frac{1}{2}}\in I$, we have
$Y\subseteq I$ since
$[N_{m-n},Y_{n+\frac{1}{2}}]=Y_{m+\frac{1}{2}}$ for all $m\in\Z$.
Moreover, $[Y_{m+\frac{1}{2}}, Y_{\frac{1}{2}}]=mM_{m+1}$ for all
$m\in\Z$ and $[Y_{1+\frac{1}{2}}, Y_{-1+\frac{1}{2}}]=2M_{1}$, so
$M\subseteq I$.

Set $\mathfrak{I_{1}}=M,\; \mathfrak{I_{2}}=M\bigoplus Y, \;
\mathfrak{I_{3}}=M\bigoplus\C N_{0}\bigoplus Y, \;
\mathfrak{I_{4}}=M\bigoplus\ N \bigoplus Y.$ Then
$I=\mathfrak{I}_{k}$ for some $k=1,2, 3,4$. Obviously,
$\mathfrak{I_{1}}$ and $\mathfrak{I_{2}}$ both have
infinite-dimensional center $M$, while the center of
$\mathfrak{I_{3}}$ and $\mathfrak{I_{4}}$ are zero, i.e.,
 $$C(\mathfrak{I_{1}})=C(\mathfrak{I_{2}})=M,\;\;\; C(\mathfrak{I_{3}})=C(\mathfrak{I_{4}})=0.$$
For any $\si\in Aut({\tsv})$, $\si(I)$ is still a non-trivial ideal
of ${\tsv}$ and $\si(C(I))=C(\si(I))$. Then
$$\si(\mathfrak{I_{i}})=\mathfrak{I_{j}}, \;\; i,j=1,2; \quad \si(\mathfrak{I}_{k})=\mathfrak{I}_{l}, \;\; k,l=3,4.$$
If $\si(\mathfrak{I_{1}})=\mathfrak{I_{2}}$, then for every
$m\in\Z$, there exists unique $x_{m}=\sum a_{m_{i}}M_{m_{i}}\in
\mathfrak{I_{1}}$ such that $\si(x_{m})=Y_{m+\frac{1}{2}}$. Then
$(m-n)M_{m+n+1}=0$ for all $m,n\in\Z$, which is impossible.
Therefore,
$$\si(\mathfrak{I_{i}})=\mathfrak{I_{i}},\;\;\; i=1,2.$$
Moreover, we obtain
\begin{equation}\label{5.1.1}
\si(M_{n})\in M, \;\;\; \si(Y_{n+\frac{1}{2}})\in M+Y.
\end{equation}
Assume that $$\si(N_{0})=\sum a_{i}M_{i}+\sum b_{j}N_{j}+\sum
c_{k}Y_{k+\frac{1}{2}},$$ where $a_{i}, b_{j}, c_{k}\in\C $.
According to (\ref{5.1.1}), $\si(M_{0})\in M$. So there exist some
$f(m)\in\C^{*}$ such that $\si(M_{0})=\sum f(m) M_{m}$. By
$\si[N_{0}, M_{0}]=[\si(N_{0}) , \si(M_{0}) ]$, we get
\begin{equation}\label{5.1.2}
\sum\limits_{m} f(m) M_{m}=\sum\limits_{m, j} b_{j}f(m) M_{m+j}.
\end{equation}
Set $p=min\{m\in\Z | f(m)\neq 0\}, \; q=max\{m\in\Z | f(m)\neq 0\}$.
If $j\neq 0$, we have
$$ p+j<p \;\;\; {\rm if }\;j<0;\;\;\; ({\rm resp.} \;q+j>q \; \; {\rm if}\; j>0).$$
By (\ref{5.1.2}), it is easy to see that $b_{j}f(p)=0 \;(\; {\rm
resp.}\; b_{j}f(q)=0 )$. So $b_{j}=0$ for all $j\neq 0$.  Then by
(\ref{5.1.2}), $b_{0}=1$. Therefore,
$$\si(N_{0})=\sum a_{i}M_{i}+N_{0}+\sum c_{k}Y_{k+\frac{1}{2}}.$$
This forces that $\si(\mathfrak{I}_{k})=\mathfrak{I}_{k},\; k=3,4.$
\end{proof}

%\newpage

\begin{lem}\label{L5.3}
For any $\si\in Aut({\tsv})$, there exist some $\tau\in
{\mathcal{I}}$  and $\epsi\in\{\pm1\}$ such that
\begin{equation}\label{E5.3.1}
\bar{\si}(L_{n})= a^{n}\epsi L_{\epsi n}+a^{n}\lambda N_{\epsi n},
\end{equation}
\begin{equation}\label{E5.3.2}
\bar{\si}(N_{n})= a^{n} N_{\epsi n},
\end{equation}
\begin{equation}\label{E5.3.3}
\bar{\si}(M_{n})=\epsi d^{2}a^{n-1}M_{\epsi(n-2\lambda)},
\end{equation}
\begin{equation}\label{E5.3.4}
\bar{\si}(Y_{n+\frac{1}{2}})=da^{n}
Y_{\epsi(n+\frac{1}{2}-\lambda)},
\end{equation}
%
%\begin{eqnarray*}
%& & \bar{\si}(L_{n})= a^{n}\epsi L_{\epsi n}+a^{n}\lambda N_{\epsi
%n},\;\;\; \bar{\si}(N_{n})= a^{n} N_{\epsi n},
%\\
%& &\bar{\si}(M_{n})=\epsi d^{2}a^{n-1}M_{\epsi(n-2\lambda)},\;\;\;
%\bar{\si}(Y_{n+\frac{1}{2}})=da^{n}
%Y_{\epsi(n+\frac{1}{2}-\lambda)},
%\end{eqnarray*}
where $\bar{\si}=\tau^{-1}\si,     \lambda\in\Z$ and $a,
d\in\C^{*}$. Conversely, if $\bar{\si}$ is a linear operator on
${\tsv}$ satisfying (\ref{E5.3.1})-(\ref{E5.3.4}) for some
$\epsi\in\{\pm1\}$, $\lambda\in\Z$ and $a, d\in\C^{*}$, then
$\bar{\si}\in Aut({\tsv})$.
\end{lem}

\begin{proof}
 By Lemma \ref{L5.2}, for all $\si\in Aut({\tsv})$,
$\si(N_{0})=\sum\limits_{i=p}^{q}
a_{i}M_{i}+N_{0}+\sum\limits_{j=s}^{t} b_{j}Y_{j+\frac{1}{2}}$ for
some $a_{i}, b_{j}\in\C$ and $p,q, s,t\in\Z.$ Let
% $$\tau=exp(-\sum\limits_{i=p}^{q}\frac{a_{i}}{2}adM_{i}+\sum\limits_{i,j=s}^{t}\frac{i-j}{4}b_{i}b_{j}adM_{i+j+1})
% exp(-\sum\limits_{j=s}^{t}b_{j}adY_{j+\frac{1}{2}}),$$
$$\tau=\prod\limits_{j=s}^{t}exp(-b_{j}{\rm{ad}} Y_{j+\frac{1}{2}})
\prod\limits_{i=p}^{q}exp(-\frac{a_{i}}{2}{\rm{ad}} M_{i})
\prod\limits_{i,j=s}^{t}exp(\frac{i-j}{4}b_{i}b_{j}{\rm{ad}}
M_{i+j+1})\in {\mathcal{I}},$$ then we can deduce that $
\si(N_{0})=\tau(N_{0})$, that is,
$$\tau^{-1}\si(N_{0})=N_{0}. $$
Set $\bar{\si}=\tau^{-1}\si$. By $[N_{0},
\bar{\si}(L_{m})]=[N_{0}, \bar{\si}(N_{m})]=0$  and  $[N_{0},
\bar{\si}(Y_{m+\frac{1}{2}})]=\bar{\si}(Y_{m+\frac{1}{2}})$  for
all $m\in\Z$, we get
$$\bar{\si}(L_{m})\in L+N,  \;\;  \bar{\si}(N_{m})\in N, \;\; \bar{\si}(Y_{m+\frac{1}{2}})\in Y.$$
 For any  $\bar{\si}\in Aut({\tsv})$, denote
$\bar{\si}|_{L}=\bar{\si}'$. By the automorphisms of the classical
Witt algebra, $\bar{\si}'(L_{m})=\epsilon a^{m}L_{\epsilon m}$ for
all $m\in\Z$, where $a\in\C^{*}$ and $\epsi\in\{\pm1\}$.
 Assume that
\begin{eqnarray*}
& &\bar{\si}(L_{0})= \epsi L_{0}+ \sum \lambda_{i} N_{i},\;\;\;
\bar{\si}(L_{n})= a^{n}\epsi L_{\epsi
n}+a^{n}\sum\lambda(n_{i})N_{n_{i}}, \;\; n\neq 0,
\\
& &\bar{\si}(N_{n})= a^{n}\sum \mu(n_{j}) N_{n_{j}},
\bar{\si}(M_{n})=a^{n}\sum f(n_{r})M_{n_{r}},
\bar{\si}(Y_{n+\frac{1}{2}})=a^{n}\sum
h(n_{t}+\frac{1}{2})Y_{n_{t}+\frac{1}{2}},
\end{eqnarray*}
where each formula is of finite terms and $\mu(n_{j}), f(n_{r}),
h(n_{t}+\frac{1}{2})\in\C^{*}$, $\lambda(i), \lambda(n_{i})\in\C$.
From $[\bar{\si}(L_{0}), \bar{\si}(M_{m})]=m\bar{\si}(M_{m})$, we
have
$$\sum \epsi m_{r}f(m_{r})M_{m_{r}}+ 2\sum \lambda_{i}f(m_{r})M_{i+m_{r}} =m\sum f(m_{r})M_{m_{r}}.$$
This forces that $\lambda_{i}=0$ for $i\neq 0$  and $\epsi
m_{r}+2\lambda_{0}=m.$ So $m_{r}=\epsi(m-2\lambda_{0})$ and
$$\bar{\si}(L_{0})= \epsi L_{0}+  \lambda_{0} N_{0},\quad
\bar{\si}(M_{n})=a^{n}
f(\epsi(n-2\lambda_{0}))M_{\epsi(n-2\lambda_{0})},$$ for all
$n\in\Z$. From   $[\bar{\si}(L_{n}), \bar{\si}(M_{0})]=0$, we get
$$ \lambda_{0}M_{\epsi n-2\epsi \lambda_{0}}=\sum \lambda_{n_{i}}M_{n_{i}-2\epsi
\lambda_{0}}. $$ Then $n_{i}=\epsi n$ and $\lambda_{\epsi
n}=\lambda_{0}$ for all $n\in\Z$. Therefore,
$$\bar{\si}(L_{n})= a^{n}\epsi L_{\epsi n}+ a^{n} \lambda_{0}N_{\epsi n},\quad \; \rm{for\; all\; n\in\Z}.$$
Since $[\bar{\si}(L_{0}), \bar{\si}(N_{n})]=n\bar{\si}(N_{n})$, we
have $ \sum (\epsi n_{j}-n)\mu(n_{j}) N_{n_{j}}=0.$ Obviously,
$n_{j}=\epsi n$ and $$\bar{\si}(N_{n})= a^{n} \mu(\epsi n)
N_{\epsi n},$$ for all $n\in\Z$, where $\mu(0)=1$. Comparing the
coefficients of $Y_{n_{t}+\frac{1}{2}}$ on the both sides of
$[\bar{\si}(L_{0}),\bar{\si}(Y_{n+\frac{1}{2}})]=(n+\frac{1}{2})\bar{\si}(Y_{n+\frac{1}{2}})$,
we obtain $n_{t}+\frac{1}{2}=\epsi(n+\frac{1}{2}-\lambda_{0})$,
which implies that $\lambda_{0}\in\Z$. So
$$\bar{\si}(Y_{n+\frac{1}{2}})=a^{n}h(\epsi(n+\frac{1}{2}-\lambda_{0}))Y_{\epsi(n+\frac{1}{2}-\lambda_{0})},
\quad \; \rm{for\; all\; n\in\Z}.$$ By
$[\bar{\si}(N_{n}),\bar{\si}(M_{m})]=2\bar{\si}(M_{m+n})$, we get
$$\mu(\epsi n)f(\epsi(m-2\lambda_{0}))=f(\epsi(m+n-2\lambda_{0})).$$
Letting  $m=2\lambda_{0}$, we obtain
$$f(\epsi n)=f(0)\mu(\epsi n).$$
 By the coefficients of $Y_{\epsi(m+n+\frac{1}{2}-\lambda_{0})}$ on
 the
 both sides of
$[\bar{\si}(N_{m}),\bar{\si}(Y_{n+\frac{1}{2}})]=\bar{\si}(Y_{m+n+\frac{1}{2}})$,
we have
$$h(\epsi(m+\frac{1}{2}))=\mu(\epsi m)h(\frac{\epsi}{2}).$$
Similarly, comparing  the coefficients  of $N_{\epsi (m+n)}$ on
the both sides of
$[\bar{\si}(L_{n}),\bar{\si}(N_{m})]=m\bar{\si}(N_{m+n})$, we have
$m\mu(\epsi m)=m\mu(\epsi (m+n))$ for all $m,n\in\Z$. Then
$$\mu(\epsi m)=\mu(0)=1,$$ for all $m\in\Z$. Therefore,
$$f(\epsi m)=f(0),\;\;\; h(\epsi(m+\frac{1}{2}))=h(\frac{\epsi}{2}).$$
Finally, we deduce that $\epsi h(\frac{\epsi}{2})^{2}=af(0)$ by
comparing the coefficient of $M_{\epsi(m+n+1-2\lambda_{0})}$ on
the
 both sides of
$[\bar{\si}(Y_{m+\frac{1}{2}}),\bar{\si}(Y_{n+\frac{1}{2}})]=(m-n)\bar{\si}(M_{m+n+1}).$
Let $d=h(\frac{\epsi}{2})$, then $f(0)=\epsi a^{-1}d^{2}$.
Therefore,
\begin{eqnarray*}
& & \bar{\si}(L_{n})= a^{n}\epsi L_{\epsi
n}+a^{n}\lambda_{0}N_{\epsi n},\;\;\; \bar{\si}(N_{n})= a^{n}
N_{\epsi n},
\\
& &\bar{\si}(M_{n})=\epsi
d^{2}a^{n-1}M_{\epsi(n-2\lambda_{0})},\;\;\;
\bar{\si}(Y_{n+\frac{1}{2}})=da^{n}
Y_{\epsi(n+\frac{1}{2}-\lambda_{0})}.
\end{eqnarray*}
It is easy to check the converse part of the theorem.

\end{proof}

Denote by $\bar{\si}(\epsi, \lambda, a, d)$ the automorphism of
$\tsv$ satisfying (\ref{E5.3.1})-(\ref{E5.3.4}), then
\begin{equation}\label{5.30}
\bar{\si}(\epsi_{1}, \lambda_{1}, a_{1}, d_{1})
\bar{\si}(\epsi_{2}, \lambda_{2}, a_{2},
d_{2})=\bar{\si}(\epsi_{1}\epsi_{2},
\epsi_{2}\lambda_{1}+\lambda_{2}, a_{1}^{\epsi_{2}}a_{2},
d_{1}d_{2}a_{1}^{\frac{\epsi_{2}-1}{2}-\epsi_{2}\lambda_{2}}),
\end{equation}
and $\bar{\si}(\epsi_{1}, \lambda_{1}, a_{1},
d_{1})=\bar{\si}(\epsi_{2}, \lambda_{2}, a_{2}, d_{2})$ if and
only if $\epsi_{1}=\epsi_{2}, \lambda_{1}=\lambda_{2},
a_{1}=a_{2}, d_{1}= d_{2}$.
Let
$$\bar{\pi}_{\epsi}=\bar{\si}(\epsi, 0, 1, 1),\quad
\bar{\si}_{\lambda}=(1, \lambda, 1, 1),\quad \bar{\si}_{a, d}=(1, 0,
a, d)$$ and
$$\mathfrak{a}=\{\bar{\pi}_{\epsi} \;|\; \epsi=\pm1 \},\quad
\mathfrak{t}=\{\bar{\si}_{\lambda} \;|\; \lambda\in\Z \},\quad
\mathfrak{b}=\{\bar{\si}_{a,d} \; |\; a,d\in\C^{*} \}. $$ By
(\ref{5.30}), we have the following relations:
$$\bar{\si}(\epsi, \lambda, a, d)=\bar{\si}(\epsi, 0,
1,1)\bar{\si}(1, \lambda, 1, 1)\bar{\si}(1, 0, a, d)\in
\mathfrak{a}\mathfrak{t}\mathfrak{b}  ,$$
$$\bar{\si}(\epsi, \lambda, a, d)^{-1}=\bar{\si}(\epsi, -\epsi\lambda, a^{-\epsi},
d^{-1}a^{\frac{1-\epsi}{2}-\lambda})  ,$$
$$\bar{\pi}_{\epsi_{1}}\bar{\pi}_{\epsi_{2}}=\bar{\pi}_{\epsi_{1}\epsi_{2}},\quad
 \bar{\si}_{\lambda_{1}}\bar{\si}_{\lambda_{2}}=\bar{\si}_{\lambda_{1}+\lambda_{2}},\quad
\bar{\si}_{a_{1},d_{1}}\bar{\si}_{a_{2},d_{2}}=\bar{\si}_{a_{1}a_{2},d_{1}d_{2}},$$
$$\bar{\pi}_{\epsi}^{-1}\bar{\si}_{\lambda}\bar{\pi}_{\epsi}=\bar{\si}_{\epsi\lambda}, \quad\;\bar{\pi}_{\epsi}^{-1}\bar{\si}_{a,d}\bar{\pi}_{\epsi}=\bar{\si}_{a^{\epsi},da^{\frac{\epsi-1}{2}}},\quad
\bar{\si}_{\lambda}^{-1}\bar{\si}_{a,d}\bar{\si}_{\lambda}=\bar{\si}_{a,da^{-\lambda}}.$$
Hence,  the following lemma holds.
\begin{lem}\label{L5.4}
$\mathfrak{a},\mathfrak{t}$ and $ \mathfrak{b} $ are all subgroups
of $Aut({\tsv})$ and
$$Aut(\tsv)={\mathcal{I}}\rtimes((\mathfrak{a}\ltimes\mathfrak{t})\ltimes\mathfrak{b}),$$
where $\mathfrak{a}\cong\Z_{2}=\{\pm 1\}, \mathfrak{t}\cong\Z$,
$\mathfrak{b}\cong\C^{*}\times\C^{*}.$ \qed
\end{lem}

Let $\C^{\infty}=\{(a_{i})_{i\in\Z}\;|\; a_{i}\in\C, {\rm{\; all
\;  but\;  \; a \;  finite \; of \;  the}} \; a_{i} \; {\rm{ \;
are \; zero}} \ \}$, ${\mathcal{I_{C}}}$ a subgroup of
${\mathcal{I}}$ generated by $\{\exp(k {\rm{ad}} M_{n})\;|\;
n\in\Z, k\in\C\}$ and
${\mathcal{\overline{I}}}={\mathcal{I}}/{\mathcal{I_{C}}}$ the
quotient group  of ${\mathcal{I}}$. Then $\C^{\infty}$ is an
abelian group and ${\mathcal{I_{C}}}$ is an abelian normal
subgroup of ${\mathcal{I}}$. As a matter of fact,
${\mathcal{I_{C}}}$ is the center of the group ${\mathcal{I}}$.

\bigskip

Note $({\rm{ad}} M_{i})^{2}=({\rm{ad}}
Y_{j+\frac{1}{2}})^{3}={\rm{ad}} M_{i}{\rm{ad}}
Y_{j+\frac{1}{2}}={\rm{ad}} Y_{j+\frac{1}{2}}{\rm{ad}} M_{i}=0$
for all $i,j\in\Z$,  then
$$exp(\a {\rm{ad}} M_{i})=1+\a {\rm{ad}} M_{i},$$
$$exp(\b {\rm{ad}} Y_{j+\frac{1}{2}})=1+\b {\rm{ad}}
Y_{j+\frac{1}{2}}+\frac{1}{2}\b^{2}({\rm{ad}}
Y_{j+\frac{1}{2}})^{2},$$
$$
exp(\a {\rm{ad}} M_{i})exp(\b {\rm{ad}} Y_{j+\frac{1}{2}})=exp(\b
{\rm{ad}} Y_{j+\frac{1}{2}})+ \a {\rm{ad}} M_{i},$$ for all $\a,
\b\in\C$. Furthermore, we get
\begin{eqnarray*}
&&exp(b_{m_{1}}{\rm{ad}}
Y_{m_{1}+\frac{1}{2}})exp(b_{m_{2}}{\rm{ad}}
Y_{m_{2}+\frac{1}{2}})\cdots exp(b_{m_{t}}{\rm{ad}}
Y_{m_{t}+\frac{1}{2}})
\\
&=&1+\sum_{k=1}^{t}b_{m_{k}}{\rm{ad}}
Y_{m_{k}+\frac{1}{2}}+\sum_{k=1}^{t}\frac{b_{m_{k}}^{2}}{2}({\rm{ad}}
Y_{m_{k}+\frac{1}{2}})^{2} +\sum_{1\leq i<j\leq
t}b_{m_{i}}b_{m_{j}}{\rm{ad}} Y_{m_{i}+\frac{1}{2}}{\rm{ad}}
Y_{m_{j}+\frac{1}{2}}
\\
&=&exp(\sum_{k=1}^{t}b_{m_{k}}{\rm{ad}}
Y_{m_{k}+\frac{1}{2}})+\frac{1}{2}\sum_{1\leq i<j\leq
t}b_{m_{i}}b_{m_{j}}({\rm{ad}} Y_{m_{i}+\frac{1}{2}}{\rm{ad}}
Y_{m_{j}+\frac{1}{2}} -{\rm{ad}} Y_{m_{j}+\frac{1}{2}}{\rm{ad}}
Y_{m_{i}+\frac{1}{2}})
\\
&=&exp(\sum_{k=1}^{t}b_{m_{k}}{\rm{ad}}
Y_{m_{k}+\frac{1}{2}})+\sum_{1\leq i<j\leq
t}\frac{m_{i}-m_{j}}{2}b_{m_{i}}b_{m_{j}}{\rm{ad}}
M_{m_{i}+m_{j}+1}
\\
&=&exp(\sum_{k=1}^{t}b_{m_{k}}{\rm{ad}} Y_{m_{k}+\frac{1}{2}})
exp(\sum_{1\leq i<j\leq
t}\frac{m_{i}-m_{j}}{2}b_{m_{i}}b_{m_{j}}{\rm{ad}}
M_{m_{i}+m_{j}+1}),
\end{eqnarray*}
 for all $m_{k}\in\Z,  b_{m_{k}}\in\C,   1\leq k \leq t$. Therefore,
\begin{equation}\label{5.32}
exp(b_{m_{1}}{\rm{ad}}
Y_{m_{1}+\frac{1}{2}})exp(b_{m_{2}}{\rm{ad}}
Y_{m_{2}+\frac{1}{2}})\cdots exp(b_{m_{t}}{\rm{ad}}
Y_{m_{t}+\frac{1}{2}}){\mathcal{I_{C}}}=exp(\sum_{k=1}^{t}b_{m_{k}}{\rm{ad}}
Y_{m_{k}+\frac{1}{2}}){\mathcal{I_{C}}}.
\end{equation}

\begin{lem}\label{L5.5}
 ${\mathcal{I_{C}}}$ and  ${\mathcal{\overline{I}}}$ are
isomorphic to $\C^{\infty}$.
\end{lem}
\begin{proof}
Define $f: {\mathcal{I_{C}}}\longrightarrow \C^{\infty}$ by
$$f(\prod\limits_{i=1}^{s}exp(\a_{k_{i}}{\rm{ad}} M_{k_{i}}))=(a_{p})_{p\in\Z},$$
 where $a_{k_{i}}=\a_{k_{i}}$ for $1\leq i\leq
s, $  and the others are zero, $k_{i}\in\Z$ and $k_{1}<
k_{2}<\cdots< k_{s}$. Since every element of ${\mathcal{I_{C}}}$
has the unique form of
$\prod\limits_{i=1}^{s}exp(\a_{k_{i}}{\rm{ad}} M_{k_{i}})$, it is
easy to check that $f$ is an isomorphism of group.

Similar to the proof above, we have ${\mathcal{\overline{I}}}
\cong \C^{\infty}$ via (\ref{5.32}).
\end{proof}

\begin{theorem}\label{T5.6}
$Aut(\tsv)=({\mathcal{I_{C}}}\rtimes{\mathcal{\overline{I}}})\rtimes
((\mathfrak{a}\ltimes\mathfrak{t})\ltimes\mathfrak{b})\cong
(\C^{\infty}\rtimes \C^{\infty} )\rtimes((\Z_{2}\ltimes\Z)\ltimes
(\C^{*}\times \C^{*}))$.

\qed
\end{theorem}

\begin{lem}\label{L5.1}{\rm{(cf. \cite{P})}}
Let  $\mathfrak{g}$  be a perfect Lie algebra and let
$\widehat{\mathfrak{g}}$ be its universal covering algebra of
$\mathfrak{g}$. Then every automorphism $\si$ of $\mathfrak{g}$
admits a unique extension $\widetilde{\si}$  to an automorphism of
$\widehat{\mathfrak{g}}$. Furthermore, the map $\si\mapsto
\widetilde{\si}$ is a group monomorphism.
\end{lem}

We will use  Lemma  \ref{L5.1} to  obtain all the automorphisms of
 ${\hsv}$ from those of
 ${\tsv}$.

 For a perfect
Lie algebra $\mathfrak{g}$, its universal covering algebra is
constucted as follows in \cite{P}. Let
$V=\Lambda^{2}\mathfrak{g}/J$, where
$$J=span\{x\wedge[y, z]+y\wedge[z, x]+z\wedge[x, y] \ | \ x,y,z\in\mathfrak{g}\}$$
is a subspace of $\Lambda^{2}\mathfrak{g}$. Then there is a natural
Lie algebra structure in the space
$\widetilde{\mathfrak{g}}=\mathfrak{g}\bigoplus V$ with the
following bracket
$$[x+u, y+v]=[x, y]+x\vee y,$$
for all $x, y\in\mathfrak{g}, u,v\in V$, where $x\vee y$ is the
image of $x\wedge y$ in $V$ under the canonical morphism
$\Lambda^{2}\mathfrak{g}\longrightarrow V$. Then the derived
algebra
$\widehat{\mathfrak{g}}=[\widetilde{\mathfrak{g}},\widetilde{\mathfrak{g}}]$
of $\widetilde{\mathfrak{g}}$ is the universal central extension of
$\mathfrak{g}$. In face, given $x\in\mathfrak{g}$ there exists $c\in
V$ such that $x+c\in\widehat{\mathfrak{g}}$.  Then the canonical
map $\widehat{\mathfrak{g}}\longrightarrow\mathfrak{g}$ is  onto
with kernel $\frak{c}\subset V$ and the resulting central
extension
$$\{0\}\longrightarrow\frak{c}\longrightarrow\widehat{\mathfrak{g}}\longrightarrow\mathfrak{g}\longrightarrow \{0\}$$
of $\mathfrak{g}$ is universal in the sense that there exists a
unique morphism from it into any other given central extension of
$\mathfrak{g}$.

 For any  $\theta\in Aut(\mathfrak{g})$, $\theta$ induces an automorphism $\theta_{V}$
 of $V$ via $$\theta_{V}(x\vee y)=\theta(x)\vee\theta(y),$$ for all $x, y\in \mathfrak{g}$. Obviously, $\theta$
extends to an automorphism $\theta_{\frak{e}}$ of
$\widetilde{\mathfrak{g}}$ by
$$\theta_{\frak{e}}(x+v)=\theta(x)+\theta_{V}(v),$$ for all $x\in \mathfrak{g},
v\in V$. By restriction, $\theta_{\frak{e}}$ induces an
automorphism $\tilde{\theta}$ of $\widehat{\mathfrak{g}}$.

In the following section, we will describe the automorphism group
of the universal central extension of ${\tsv}$ using the above
method. Firstly, we have the following  lemmas.

%Let $V=\Lambda^{2}({\tsv})/J$, where $J$ is the
%subspace of $\Lambda^{2}({\tsv})$ spanned by
%$x\wedge[y, z]+y\wedge[z, x]+z\wedge[x, y]$ for all
%$x,y,z\in{\tsv}$.

\begin{lem}\label{L5.6}
In $V=\Lambda^{2}({\tsv})/J$,  we have the following relations for
all $m,n\in\Z$ :
\begin{eqnarray*}
& &L_{m}\vee L_{n}=\frac{n-m}{m+n}L_{0}\vee
L_{m+n},\quad\;m+n\neq0;
\\
& &L_{m}\vee L_{-m}=\frac{m^{3}-m}{6}L_{2}\vee L_{-2};
\\
& &L_{m}\vee N_{n}=\frac{n}{m+n}L_{0}\vee N_{m+n}, \quad m+n\neq
0;
\\
& &L_{m}\vee N_{-m}=\frac{m^{2}+m}{2}(L_{1}\vee N_{-1}+L_{-1}\vee
N_{1})-mL_{-1}\vee N_{1};
\\
& &L_{m}\vee M_{n}=\frac{n}{2}N_{0}\vee M_{m+n};\quad\; L_{m}\vee
Y_{n+\frac{1}{2}}=(n+\frac{1-m}{2})N_{0}\vee Y_{m+n+\frac{1}{2}};
\\
& &N_{m}\vee N_{n}=m\delta_{m+n,0}N_{1}\vee N_{-1};\; N_{m}\vee
M_{n}=N_{0}\vee M_{m+n};\; N_{m}\vee Y_{n+\frac{1}{2}}=N_{0}\vee
Y_{m+n+\frac{1}{2}};
\\
& &M_{m}\vee M_{n}=M_{m}\vee Y_{n+\frac{1}{2}}=0;\quad\;
Y_{m+\frac{1}{2}}\vee Y_{n+\frac{1}{2}}=\frac{m-n}{2}N_{0}\vee
M_{m+n+1}.
\end{eqnarray*}
\qed
\end{lem}

Using Lemma \ref{L5.6}, we have the following result.

\begin{lem}\label{L5.7} The universal central extension of
${\tsv}$,  denoted  by ${\hsv}$, has a basis
\\
$\{L'_{n}, M'_{n}, N'_{n}, Y'_{n+\frac{1}{2}}, C_{L}, C_{LN},
C_{N}\ |\ n\in\Z\}$
 with the
following products:
$$[L'_{m}, L'_{n}]=(n-m)L'_{m+n}+\delta_{m+n,0}\dis\frac{m^{3}-m}{12}C_{L},$$
$$[N'_{m}, N'_{n}]=n\delta_{m+n,0}C_{N},$$
$$[L'_{m}, N'_{n}]=nN'_{m+n}+\delta_{m+n,0}(n^{2}-n)C_{LN},  $$
$$[M'_{m}, M'_{n}]=0, \ \ [Y'_{m+\frac{1}{2}}, Y'_{n+\frac{1}{2}}]=(m-n)M'_{m+n+1},$$
$$[L'_{m}, M'_{n}]=nM'_{m+n},  \ \ [L'_{m},Y'_{n+\frac{1}{2}}]=(n+\dis\frac{1-m}{2})Y'_{m+n+\frac{1}{2}}, $$
$$[N'_{m}, M'_{n}]=2M'_{m+n},  \ \ [N'_{m},Y'_{n+\frac{1}{2}}]=Y'_{m+n+\frac{1}{2}}, \ \ [M'_{m},Y'_{n+\frac{1}{2}}]=0,$$
$$[{\hsv}, C_{L}]=[{\hsv}, C_{LN}]=[{\hsv}, C_{N}]=0,$$
where
$$L'_{0}=L_{0}, \quad\;  N'_{0}=N_{0}+L_{-1}\vee N_{1};$$
$$L'_{m}=L_{m}+\frac{1}{m}L_{0}\vee L_{m},\;\;N'_{m}=N_{m}+\frac{1}{m}L_{0}\vee N_{m},\;\;\; m\neq 0;$$
$$M'_{n}=M_{n}+\frac{1}{2}N_{0}\vee M_{n},\quad\;
Y'_{n+\frac{1}{2}}=Y_{n+\frac{1}{2}}+N_{0}\vee
Y_{n+\frac{1}{2}},\quad\; n\in\Z;$$
$$C_{L}=2L_{2}\vee L_{-2}, \;\;\; C_{LN}=\frac{1}{2}(L_{1}\vee N_{-1}+L_{-1}\vee
N_{1}), \;\;\;C_{N}=N_{-1}\vee N_{1}.$$
\end{lem}\qed

\begin{lem}\label{L5.10}
For any  $\tilde{\theta}\in Aut({\hsv})/{\mathcal{I}}$, we have
\begin{equation}\label{e5.5.1}
\tilde{\theta}(L'_{n})=a^{n}\epsi L'_{\epsi n}+ a^{n}\lambda
N'_{\epsi
n}-\lambda\delta_{n,0}C_{LN}+\dis\frac{\epsi}{2}\lambda^{2}\delta_{n,0}C_{N},
\end{equation}
\begin{equation}\label{e5.5.2}
\tilde{\theta}(N'_{n})=a^{n}N'_{\epsi
n}+(\epsi-1)\delta_{n,0}C_{LN}+\epsi \lambda \delta_{n,0} C_{N},
\end{equation}
\begin{equation}\label{e5.5.3}
\tilde{\theta}(M'_{n})=\epsi d^{2}a^{n-1}M'_{\epsi(n-2\lambda)},
\end{equation}
\begin{equation}\label{e5.5.4}
\tilde{\theta}(Y'_{n+\frac{1}{2}})=d
a^{n}Y'_{\epsi(n+\frac{1}{2}-\lambda)},
\end{equation}
\begin{equation}\label{e5.5.5}
\tilde{\theta}(C_{L})={\epsi} C_{L},\;
\tilde{\theta}(C_{LN})={\epsi} C_{LN},\;
\tilde{\theta}(C_{N})={\epsi} C_{N},
\end{equation}
for all $n\in\Z$, where $\epsilon\in\{\pm1\}$, $\lambda\in\Z,$ $a,
d\in\C^{*}$. Conversely, if $\tilde{\theta}$ is a linear operator
on ${\hsv}$ satisfying (\ref{e5.5.1})-(\ref{e5.5.5}) for some
$\epsilon\in\{\pm1\}$, $\lambda\in\Z,$ $a, d\in\C^{*},$
 then $\tilde{\theta}\in
Aut({\hsv})$.

\qed
\end{lem}

From the above lemmas  and Theorem \ref{T5.6}, we obtain the last
main theorem.

\begin{theorem}\label{T5.10}
$Aut(\tsv)\cong Aut(\hsv).$
\end{theorem}

\bibliography{}

\end{document}